\documentclass[11pt, a4paper,reqno]{amsart}
\usepackage[margin=25mm]{geometry}
\usepackage[dvipsnames]{xcolor} 
\definecolor{Blue}{rgb}{0,0,1}
\definecolor{Green}{RGB}{0,100,0}
\definecolor{Pink}{RGB}{153,51,255}
\newcommand{\blue}{\textcolor{Blue}}
\newcommand{\red}{\textcolor{Red}}
\newcommand{\green}{\textcolor{Green}}
\newcommand{\orange}{\textcolor{Orange}}
\newcommand{\pink}{\textcolor{Pink}}
\usepackage{booktabs,bm,caption,subcaption}
\usepackage{hyperref}
\usepackage[capitalize,nameinlink,noabbrev]{cleveref}
\usepackage{graphicx}
\graphicspath{{Pictures/}}
\usepackage{amsfonts}
\usepackage[alphabetic, initials, nobysame]{amsrefs}
\usepackage{epsf}
\usepackage{amssymb}
\usepackage{amsmath}
\usepackage{amscd}
\usepackage{amsthm}
\usepackage{tikz}
\usetikzlibrary{
  cd,
  shapes,
  calc,
  positioning,
  arrows,
  decorations.pathreplacing,
  decorations.markings,
}
\usepackage{pdfpages}
\usepackage{fancyhdr}
\usepackage{setspace}
\usepackage[all]{xy}
\usetikzlibrary{matrix}
\usepackage{verbatim}
\usepackage{enumerate}
\usepackage{mathrsfs}
\usepackage{pinlabel}
\usepackage{lscape}

\usepackage[normalem]{ulem}

\newtheorem{theorem}{Theorem}[section]

\newtheorem{beastmate}{\cref{thm:beast}}
\crefname{beastmate}{\cref{thm:beast}}{\cref{thm:beast}}

\newtheorem{proposition}[theorem]{Proposition}

\newtheorem{corollary}[theorem]{Corollary}
\newtheorem{lemma}[theorem]{Lemma}

\newtheorem*{claim}{Claim}

\newtheorem*{rep@theorem}{\rep@title}
\newcommand{\newreptheorem}[2]{
\newenvironment{rep#1}[1]{
\def\rep@title{#2 \ref{##1}}
\begin{rep@theorem}}
{\end{rep@theorem}}}
\newreptheorem{theorem}{Theorem}
\newreptheorem{lemma}{Lemma}
\newreptheorem{proposition}{Proposition}
\newreptheorem{corollary}{Corollary}

\theoremstyle{definition}
\newtheorem{definition}[theorem]{Definition}
\newtheorem{example}[theorem]{Example}
\newtheorem{remark}[theorem]{Remark}
\newtheorem{problem}[theorem]{Problem}
\newtheorem{notation}[theorem]{Notation}
\newtheorem{convention}[theorem]{Convention}
\newtheorem{case}{Case}
\newtheorem{step}{Step}

\numberwithin{equation}{section}
\numberwithin{figure}{section}
\numberwithin{table}{section}

\newcommand{\GL}{\mathrm{GL}}

\newcommand{\Z}{\mathbb{Z}}

\newcommand{\CP}{\mathbb{CP}}

\def\Z{\mathbb{Z}}

\newcommand{\Id}{\operatorname{Id}}

\def\sm{\smallsetminus}

\DeclareMathOperator\Arf{Arf}

\def\lk{\operatorname{lk}}

\def\ol#1{\overline{#1}{}}
\def\wt#1{\widetilde{#1}{}}

\makeatletter
\def\sn{\mathrm{sn}}
\def\gt{g^\mathrm{top}_4}

\renewcommand{\top}{\mathrm{top}}
\newcommand{\ucp}{u_{\CP^2}}
\newcommand{\fr}{\mathrm{fr}}
\newcommand{\ucpbar}{u_{\overline{\CP^2}}}
\DeclareMathOperator{\sgn}{\text{sgn}}

\begin{document}
\title{Slicing knots in definite $4$-manifolds}

\author[{A.\ Kjuchukova}]{Alexandra Kjuchukova}
\address{Department of Mathematics, University of Notre Dame, Notre Dame, IN 46556 USA }
\email{akjuchuk@nd.edu}
\urladdr{https://math.nd.edu/people/faculty/alexandra-kjuchukova/}

\author[A.\ N.\ Miller]{Allison N.\ Miller}
\address{Department of Mathematics \& Statistics, Swarthmore College, 500 College Avenue, Swarthmore, PA. 19081}
\email{amille11@swarthmore.edu }
\urladdr{https://sites.google.com/view/anmiller/}

\author[A.\ Ray]{Arunima Ray}
\address{Max-Planck-Institut f\"{u}r Mathematik, Vivatsgasse 7, 53111 Bonn, Germany}
\email{aruray@mpim-bonn.mpg.de }
\urladdr{http://people.mpim-bonn.mpg.de/aruray/}

\author[S.\ Sakall{\i}]{S\"umeyra Sakall{\i}}
\address{Department of Mathematical Sciences, University of Arkansas, Fayetteville, AR 72701, USA}
\email{ssakalli@uark.edu}
\urladdr{https://sites.google.com/umn.edu/ssakalli/home}

\def\subjclassname{\textup{2020} Mathematics Subject Classification}
\expandafter\let\csname subjclassname@1991\endcsname=\subjclassname
\subjclass{
57K10, 
57N35, 
57N70, 
57R40.  
}

\begin{abstract} 
We study the \emph{$\CP^2$-slicing number} of knots, i.e.\ the smallest $m\geq 0$ such that a knot $K\subseteq S^3$ bounds a properly embedded, null-homologous disk in a punctured connected sum~$(\#^m\CP^2)^{\times}$.  We give a lower bound on the smooth $\CP^2$-slicing number of a knot in terms of its double branched cover, and we find knots with arbitrarily large but finite smooth $\CP^2$-slicing number. We also give an upper bound on the topological $\CP^2$-slicing number in terms of the Seifert form and find knots for which the smooth and topological $\CP^2$-slicing numbers are both finite, nonzero, and distinct. 
\end{abstract}

\maketitle

\section{Introduction}
\label{sec:introduction}

The study of slice knots was initiated in the 1950s by Fox and Milnor~\cite{fox-milnor} in the context of resolving singularities of surfaces embedded in $4$-manifolds. Since then, numerous connections have been established between sliceness and fundamental open problems in $4$-manifold topology, including the smooth Poincar\'e conjecture~\cite{man-and-machine} and the exactness of the surgery sequence for topological $4$-manifolds~\cite{casson-freedman-atomic} (see also~\cite{DET-book-enigmata}).  
Recent work~\cites{manolescu-marengon-piccirillo,mmsw,manolescu-piccirillo} indicates that slicing knots not only in $B^4$ but in more general definite $4$-manifolds may answer long-standing questions about the existence of exotic smooth structures in dimension four.
As an example, \cite{manolescu-piccirillo}*{Theorem~1.4} provides a list of 23 knots, such that if any of them bounds a smooth, properly embedded, null-homologous disk in $(\#^m\CP^2)\sm \mathring{B}^4$, for some $m$, then there exists an exotic smooth structure on $\#^m\CP^2$. 
Obstructing the sliceness of knots in progressively closer approximations of $B^4$ reveals intricate structure within the knot concordance group~\cites{COT1, COT2, CT, CHL09, CHL11}, and, in the case of approximations using definite $4$-manifolds, can distinguish between smooth concordance classes of topologically slice knots~\cites{CHH,cochran-horn,cha-kim-bipolar}. 

A knot $K\subseteq S^3$ is said to be \emph{slice} in a closed, smooth $4$-manifold $M$ if $K\subseteq \partial (M\setminus \mathring{B}^4)$ bounds a smooth, properly embedded disk in $M\setminus \mathring{B}^4$, and \emph{$H$-slice} if the disk can be chosen to be null-homologous. Of course, for $M=B^4$ the two notions coincide. 
The \emph{$\CP^2$-slicing number} of a knot $K$, denoted by $\ucp(K)$, is either the smallest $m\geq 0$ such that $K$ is $H$-slice in $\#^m \CP^2$,  or $\infty$ if no such $m$ exists. We define the $\overline{\CP^2}$-slicing number analogously and denote it by $\ucpbar(K)$. When $\ucp(K)<\infty$ (resp. $\ucpbar(K)<\infty$), we say that $K$ is \emph{positively slice} (resp. \emph{negatively slice}). 
Smoothly slice knots have $\ucp$ and $\ucpbar$ equal to zero.  In the converse direction, there exist knots with both $\ucp$ and $\ucpbar$ infinite~\cite{sato-definite}, including topologically slice knots. 

If a knot is $H$-slice in a closed, smooth, oriented $4$-manifold with positive definite intersection form, then its signature function is non-positive~\cite{CHH}*{Proposition~4.1}, so there exist knots, such as the right-handed trefoil, which have finite $\ucp$ but infinite $\ucpbar$ (\cref{ex:infinite-on-one-side}). In fact, by a version of the classical Murasugi-Tristram inequality (see~\cref{sec:background}) the values of the signature function give a lower bound on $\ucp$,  so there exist knots with arbitrarily large, but finite, $\ucp$ (\cref{ex:CNbound-2bridge}), though these necessarily have infinite $\ucpbar$. We find a new lower bound on $\ucp$ that can be applied to knots with trivial signature function,  in terms of the double branched cover. 

\begin{theorem}\label{thm:beast}
Let $K\subseteq S^3$ be a knot with $\sigma(K)=0$. Suppose  $K$ is $H$-slice in $\#^m\CP^2$ for some $0\leq m< \infty$. Then $\Sigma_2(K)$, the double
cover of $S^3$ branched along $K$, bounds a compact, smooth, oriented $4$-manifold $X$ with $b_2(X) =2m$, whose intersection form is positive definite and of half-integer surgery type.
\end{theorem}

Half-integer surgery type pairings are defined in \cref{def:half-int}. The manifold $X$ in the statement above is obtained as the double cover of $(\#^m\CP^2)\setminus \mathring{B}^4$ branched along the putative slice disk.  Our result can be seen as an extension of work of Owens on the unknotting number~\cites{Owens08, Owensslicing}, specifically~\cite{Owensslicing}*{Theorem~2}.   Owens leverages the fact that if a knot $K$ can be changed to the unknot (or, more generally a slice knot) by $m$ positive to negative crossing changes then $K$ is $H$-slice in $\#^m \CP^2$ in a particularly nice way (see~\cref{sec:background}).  However,  general $H$-sliceness in $\#^m \CP^2$ can be more complicated, involving both an initial concordance and generalized crossing changes, as described in \cref{lem:slicetotwist}. A generalization of~\cites{Owens08, Owensslicing} to slicing numbers of knots occurs in the work of Owens-Strle~\cite{owens-strle}*{Theorem~1} (see also~\cite{nagel-owens}*{Theorem~1.1}). In Theorem~\ref{thm:beast}, we extend Owens' ideas to the yet more general setting of $\CP^2$-slicing numbers. 

We extract the first intrinsically smooth lower bounds for $\ucp$ and $\ucpbar$  by combining \cref{thm:beast} with Donaldson's Theorem~A~\cite{donaldson}, which may be leveraged to obstruct $\Sigma_2(K)$ from bounding a definite manifold of the half-integer surgery type. Specifically we provide the first examples of knots with arbitrarily large $\ucp$ and trivial signature function, and the first examples of knots where both $\ucp$ and $\ucpbar$ are finite and arbitrarily large. 

\begin{theorem}\label{thm:finite-and-big}\label{thm:big-and-precise}
For any $n\geq 0$, there exists a knot $K$ such that $n\leq \ucp(K)< \infty$ and $n\leq \ucpbar(K)< \infty$. 
Also, for any $n\geq 0$ there exists a knot $J$ with trivial signature function and $\ucp(J)=n$.
\end{theorem}

The knots used to prove \cref{thm:finite-and-big} are connected sums of twist knots.  Their key relevant property is that the double branched cover bounds both a positive definite and a negative definite $4$-manifold, which we can then cap off with the manifolds provided by~\cref{thm:beast} to get effective lower bounds on $\ucp$ and $\ucpbar$ simultaneously.  For just the second sentence of~\cref{thm:finite-and-big}, we may also use pretzel knots, for example, we will show in \cref{cor:bigstrand} that for the $(2k+1)$-strand pretzel knot $J=P(p,-p-2,p,-p-2,\dots,p)$ for odd $p\geq 3$, we have $\ucp(J)=J$. 

In order to provide further context for Theorem~\ref{thm:finite-and-big}, we briefly discuss why other knot invariants are not suitable for proving such a result. 
First, a  knot $K$ as above must have trivial signature function, and therefore the first statement in \cref{thm:finite-and-big} is inaccessible through signature bounds.  Additionally, while tools from gauge theory and Khovanov homology can be used to obstruct $H$-sliceness (as well as sliceness) in smooth, compact, oriented, definite $4$-manifolds, these are generally only in terms of the \textit{signs} of the invariants, not their values. For example, if a knot is $H$-slice in a positive definite $4$-manifold, its Heegaard-Floer $\tau$-invariant and Rasmussen's $s$-invariant are non-negative~\citelist{\cite{OS-fourball}*{Theorem~1.1}\cite{mmsw}*{Corollary~1.9}} (see~\citelist{\cite{CHH}*{Proposition~1.2}\cite{manolescu-piccirillo}*{Section~2}} for further such obstructions). On the other hand, for $K$ the $(p,1)$ cable of the right-handed trefoil, observe that $\ucp(K)=1$ (\cref{ex:satellite-slice}), despite the fact that~$\tau(K)=p$ may be large~\cite{hedden806knot}*{Theorem~1.2}.
Bauer-Furuta type invariants were used to obstruct sliceness in symplectic $4$-manifolds in \cite{bauer-furuta}; those techniques do not apply to our situation since neither $m\CP^2$ for $m>1$ nor $n\ol{\CP^2}$ for $n\geq 1$ is symplectic. 
Finally,  the aforementioned work of Owens~\cites{Owens08,Owensslicing} and Owens-Strle~\cite{owens-strle} effectively used Heegaard-Floer $d$-invariants to bound the (signed) unknotting and slicing numbers of knots.  Roughly speaking, they proved an analogue of \cref{thm:beast} in which one also concludes that the determinant of the intersection form $Q_X$ must divide the determinant of the knot $K$.  This allows one to restrict to a finite list of possible intersection forms,  each of which would impose constraints on the $d$-invariants of the double branched cover of $K$. However,  in the setting of $\CP^2$-slicing numbers, there is no such straightforward restriction on the determinant of $K$: a concordance between two knots can arbitrarily increase the determinant, while of course preserving the $\CP^2$-slicing number.

\subsection{\texorpdfstring{Topological $\CP^2$-slicing numbers}{Topological CP2 slicing numbers}} The quantities $\ucp$ and $\ucpbar$ both have purely topological counterparts. 
As above, we say that a knot $K\subseteq S^3$ is \emph{topologically slice} in a closed, topological $4$-manifold $M$ if $K\subseteq \partial (M\setminus \mathring{B}^4)$ bounds a locally flat, properly embedded disk in $M\setminus \mathring{B}^4$, and \emph{topologically $H$-slice} if the disk is further null-homologous.  
The \emph{topological $\CP^2$-slicing number} of a knot $K\subseteq S^3$, denoted by $\ucp^\top(K)$, is the smallest $m\geq 0$ such that $K$ is topologically $H$-slice in $\#^m \CP^2$. We define the topological $\overline{\CP^2}$-slicing number analogously and denote it by $\ucpbar^\top(K)$. The following gives an upper bound on $\ucp^\top$.

\begin{theorem}\label{thm:upperbound}
Let $K\subseteq S^3$ be a knot with Seifert matrix~$A$, with respect to some Seifert surface~$F$ and choice of generators~$\alpha_1, \dots, \alpha_{2g}$ for $H_1(F;\Z)$, where $g$ is the genus of $F$.

If there exists an integral $2g \times 2g$ matrix $B$ with $\det (tB-B^T)=\pm t^k$, for some $k$, and integers $\{c_{i,j}\}$ for $i=1,\cdots, n$ and $j=1,\dots, g$, such that $A$ can be decomposed as the difference 
\begin{align*}
A_{2g\times 2g}= B_{2g\times 2g}
- \sum_{i=1}^n \left[
 \begin{array}{cccc}
c_{i,1}^2 & c_{i,1}c_{i,2} & \dots &c_{i,1} c_{i,2g} \\
c_{i,1} c_{i,2}  & c_{i,2}^2 & \ddots & c_{i,2} c_{i,2g}\\
\vdots & \ddots & \ddots & \vdots \\
c_{i,1} c_{i,2g} & c_{i,2} c_{i,2g} & \dots & c_{i,2g}^2
\end{array}
\right],
\end{align*}
then $\ucp^\top(K) \leq n$. 
\end{theorem}

The seemingly technical condition on the Seifert matrix $A$ above comes from a natural geometric construction. The idea is that, under the given condition on $A$, there exists an $n$-component unlink $\gamma:= (\gamma_1 \sqcup \cdots\sqcup \gamma_n)$ in the complement of the Seifert surface $F$ for $K$ in $S^3$, such that the image of $K$ in the copy of $S^3$ produced by performing simultaneous $+1$-framed Dehn surgeries on $S^3$ along every~$\gamma_i$ has Alexander polynomial one, and is therefore topologically slice in $B^4$. In other words, $K$ is $n$  generalized positive crossing changes away from a topologically slice knot.

\subsection{Smooth vs topological slicing} We next consider the relationship between $\ucp^\top$ and $\ucp$, demonstrating how to use \cref{thm:upperbound} in practice. For any knot $K$ which is topologically but not smoothly slice in $B^4$ we have that $\ucp^\top(K)=0$ but $0\neq \ucp(K)\leq \infty$. 
For instance, let $K$ be the positive clasped untwisted Whitehead double of the right-handed trefoil knot. Since $K$ has Alexander polynomial one, it is topologically slice~\cite{FQ}*{Theorem~11.7B}, and since it can be unknotted by a single positive to negative crossing change, we have $\ucp(K)\leq 1$ (see \cref{sec:background} for further details and references). Moreover, $K$ is not smoothly slice (see e.g.\ ~\cite{Gom86}). Therefore, $\ucp(K)=1$ and $\ucp^\top(K)=0$. We find the first examples of knots where $\ucp$ and $\ucp^\top$ are finite, distinct and moreover simultaneously nonzero.

\begin{corollary}\label{thm:top-smooth-distinct}
Let $p\geq 3$ be odd.
\leavevmode
\begin{enumerate}
\item The pretzel knot $K_{p,2}:=P(p, -p-4, 3p+14)$ has $\ucp^\top(K_p)=1$ and $\ucp(K)=2$. Moreover, the elements of $\{K_{p,2}\}_p$ are distinct in topological concordance.  
\item The pretzel knot $K_{p,3}:=P(p, -p-6, 3p+22)$ has $\ucp^\top(K_p)=1$ and $\ucp(K)=3$. Moreover, the elements of $\{K_{p,3}\}_p$ are distinct in topological concordance.   
\end{enumerate}
\end{corollary}

We expect that the gap between $\ucp$ and $\ucp^\top$ can be arbitrarily large.  It would be particularly interesting to find topologically slice knots with arbitrarily large and finite $\ucp$ and $\ucpbar$. For example, let $K$ denote the positively clasped untwisted Whitehead double of the right-handed trefoil.
As mentioned above, we know that $\ucp^\top(K)=0$ and $\ucp(K)=1$. 
We conjecture that $\ucp(\#^nK)=n$.  Note that $\ucpbar(\#^nK)=\infty$ for all $n\geq 1$, since for example $\tau(\#^n K)=n>0$.

\subsection{Connection with other invariants}

As mentioned above, the $\CP^2$-slicing number is bounded above by well-known knot invariants such as the unknotting and slicing numbers. Specifically, if a knot $K$ can be unknotted (or smoothly sliced) by $k$ positive to negative crossing changes then $\ucp(K)\leq k$. We remark that in all of our computations of $\ucp$ we obtain upper bounds by crossing changes, though our lower bounds are rather on the number of \emph{generalized crossings}; see \cref{sec:background} for further details and references.

Recall that the \emph{stablizing number} $\sn(K)$ for a knot $K$ with $\Arf(K)=0$ is the smallest $m\geq 0$ such that $K$ is $H$-slice in $\#^m(S^2\times S^2)$. We have that $\sn(K)\leq g_4(K)$ and $\sn^\top(K)\leq \gt(K)$, as shown in~\cite{CN}.\footnote{Conway-Nagel give a proof for the second inequality, but their argument also applies in the smooth setting.} Unlike the stabilizing number, there is no reason to expect that $\ucp$  will be restricted by the slice genus or even the Seifert genus. We show in~\cref{cor:precise3strand}(3) that for odd $p\geq 3$ and $r> p+6$, the $3$-strand pretzel knot $P(p,-p-6,r)$ has $\ucp(P(p,-p-6,r))=3$, while it has Seifert genus one. For $k\geq 0$, odd $p\geq 3$, and $r> p+2k$, we speculate that $\ucp(P(p,-p-2k,r))=k$. At present we do not have a lower bound strong enough to show this. Note that the values of many classical knot invariants are bounded by (some linear function of) the Seifert genus or slice genus, indicating the inherent difficulty in finding effective lower bounds on $\ucp$. Illustrating this phenomenon is the fact that there are no known examples of knots with high unknotting number relative to their Seifert genus~\cite{lewarkMO}. 

In contrast to the situation for $\ucp$, in the topological category, we have the following result. 

\begin{proposition}\label{prop:genus-bound}
Let $K$ be a knot with Seifert genus one. 
If $\sigma(K)=2$, then $\ucp^\top(K)= \infty$.  Otherwise, $\ucp^\top(K) \leq 4$. 
\end{proposition}

When $K$ has genus greater than one,  we do not expect the vanishing of Tristram-Levine signatures to guarantee that $\ucp^{\top}(K)$ is finite,  since Casson-Gordon signatures should give an additional obstruction (see Section~\ref{sec:background}). However, it is certainly possible that given a knot $K$,  either $\ucp^{\top}(K)$ is infinite or it is bounded above by some function of the Seifert genus of $K$. 

\subsection{Some open questions}
We finish the introduction by collecting some open questions about $\CP^2$-slicing numbers.  
First we note that while \cref{thm:finite-and-big} produces knots for which both $\ucp$ and $\ucpbar$ are finite and arbitrarily large, we have so far been unable to pin down both values precisely. 

\begin{problem}
For each $m, n \geq 0$, find a knot $K$ such that $\ucp(K)=m$ and $\ucpbar(K)=n$. 
\end{problem}

\cref{thm:top-smooth-distinct} shows that the values of $\ucp^\top$ and $\ucp$ can be distinct for a fixed knot. We expect that the gap between the two quantities can be arbitrarily large. 

\begin{problem}\label{prob:top-smooth}
For each $0\leq m \leq n \leq \infty$, find a knot $K$ with   $\ucp^\top(K)=m$ and $\ucp(K)=n$. 
\end{problem}

Note that the case above when $m=n$ is resolved by the torus knot $T_{2,2m+1}$ (or $-T_{2,3}$ for $m=n=\infty$). 

Finally, we showed in \cref{cor:precise3strand} that there exists a knot $K$ with Seifert genus one and $\ucp(K)=3$. We speculate that there might exist knots with genus one and arbitrarily large but finite $\ucp$. We showed in \cref{prop:genus-bound} that the analogous phenomenon does not occur for $\ucp^\top$.

\begin{problem}\label{prob:genus-one}
Are there knots with Seifert genus one and arbitrarily large but finite $\ucp$?
\end{problem}
The examples of \cref{cor:precise3strand} suggest candidates for an affirmative answer: it is possible that for $p$ odd, $k \in \mathbb{N}$, and $r>p+2k$, the $\CP^2$ slicing number $\ucp(P(p, -p-2k, r))$ is exactly $k$.  A similar question asks about $\ucpbar$ of pretzel knots of this form: it is easy to observe that this number is bounded above by $(r-p-2k)/2$, but 
finding a lower bound
is challenging.  We also note that the known lower bounds on $\ucp^\top$  coming from Tristram-Levine and Casson-Gordon signatures~\cite{CN}  are ill-suited to addressing this question, since they are roughly controlled by the $3$-genus of a knot. 

We finish this section with a natural generalization of the questions we address in this paper. 

\begin{problem}\label{prob:indefinite}
Given a knot $K$, denote
\[
X(K) := \{ (m,n)\in \Z_{\geq 0}\times \Z_{\geq 0} \mid   K \text{ is } H\text{-slice in } m\CP^2 \#  n\ol{\CP^2}\}.
\]
Two natural problems arise.  First,  for a fixed knot of interest $K$,  determine $X(K)$ precisely.  Second,  characterize which sets may occur as $X(K)$ for some knot $K$. 
\end{problem}

According to \cites{livingston-slicing, livingston-nullhomologous}, every knot with Seifert genus one can be converted to the unknot via two generalized crossing changes (one positive to negative and one negative to positive), and hence is $H$-slice in $\CP^2 \# \overline{\CP}^2$. We give some new examples of knots $K$, specifically $3$-strand pretzel knots, for which we determine the set $X(K)$ precisely, in~\cref{cor:indefinite}. In general, it is difficult to effectively obstruct sliceness in indefinite $4$-manifolds, compared to definite ones. For example, the version of \cref{thm:top-smooth-distinct} for slicing knots in connected sums of $S^2\times S^2$~ \cite{topdiffstabilising}*{Theorem~1.8} necessitated the development of a Seiberg-Witten $K$-theory for $3$-manifolds with involution (see also~\cite{manolescu-marengon-piccirillo}). 
 Even less is known about the second part of Problem~\ref{prob:indefinite}, besides the immediate  observation that such a set must have the property that whenever $(m,n)$ is in $X(K)$ for some $K$ then  $(m+1, n),(m,n+1)\in X(K)$ as well.  Both parts of this problem could equally well be posed in the topological category. 

\subsection*{Outline}
	\cref{sec:background} reviews some background and elementary observations. \cref{sec:beast} contains the proof of \cref{thm:beast}. In \cref{sec:alt} we apply \cref{thm:beast} to alternating knots and prove \cref{thm:finite-and-big}. In \cref{sec:pretzels} we apply \cref{thm:beast} to pretzel knots. In~\cref{sec:top-v-smooth} we prove \cref{thm:upperbound,prop:genus-bound}.

\subsection*{Conventions} All manifolds are assumed to be compact, connected, and oriented. Knots are assumed to be oriented, unless specifically mentioned otherwise.  Given a knot $K$, we refer to its mirror image as $\ol{K}$ and its reverse as $rK$. The concordance inverse is denoted $-K$. 
We denote the standard inner product on $\Z^N$ by $\langle \cdot ,\cdot\rangle$.
	
\subsection*{Acknowledgements}
We would like to thank Anthony Conway, Chuck Livingston, and Brendan Owens for helpful comments on a draft of this paper. 
A large portion of this work was carried out while AK, AR, and SS were at the Max Planck Institute for Mathematics. We are grateful to the MPIM for bringing us together and supporting our work. ANM was supported by NSF grant DMS-1902880 and is welcome at MPIM anytime.

\section{Background and elementary observations}\label{sec:background}

Recall that the \emph{unknotting number} $u(K)$ (resp.\ the \emph{slicing number} $u_s(K)$) of a knot $K\subseteq S^3$ is the least number of crossing changes needed to transform $K$ into an unknot (resp.\ a slice knot). If $K$ can be changed to the unknot (or a slice knot) by $m$ positive to negative crossing changes, then it is $H$-slice in $\#^m\CP^2$. To see this, take the track of the homotopy, blow up each double point with a $\CP^2$, resolve each of the double points, and cap off with a slice disk. (For more details on this construction see~\citelist{\cite{Cochran-Lickorish}*{Lemma~3.4}\cite{manolescu-piccirillo}*{Lemma~2.3}}.) 
The least number of positive to negative crossing changes needed to change $K$ to the unknot (resp.\ a slice knot) is called the \emph{positive unknotting number} $u^+(K)$ (resp.\ the \emph{positive slicing number} $u^+_s(K)$). We see then that 
\[
\ucp(K) \leq u_s^+(K) \leq u^+(K)
\]
for every knot $K$. 

Since a smooth $H$-slicing disk is in particular a topological $H$-slicing disk, we also have that 
\[
\ucp^\top(K)\leq \ucp(K)
\]
for every knot $K\subseteq S^3$. By changing orientations, we see that $\ucp(K)=\ucpbar(\ol{K})$ and $\ucp^\top(K)=\ucpbar^\top(\ol{K})$. Of course $\ucp(K)=\ucp(rK)$ and $\ucp^\top(K)=\ucp^\top(rK)$. Putting this together we see that all the flavors of $\CP^2$-slicing numbers of $\ol{K}$ and $-K$ coincide. While a knot's orientation does not influence the value of $\ucp$ or $\ucp^\top$, we include it in diagrams for concreteness, and we usually work with $-K$ rather than $\ol{K}$. In general, $\ucp(K)$ and $\ucpbar(K)$ are unrelated, in either category, as we see in the following example. 

\begin{example}\label{ex:infinite-on-one-side}
Let $K$ be the right-handed trefoil. Then $\ucp^\top(K)=\ucp(K)=1$, since $K$ can be changed to the unknot by a positive to negative crossing change, and it is not smoothly (nor topologically) slice. However, since $\sigma(K)=-2$, it follows that $\ucpbar^\top(K)=\ucpbar(K)=\infty$  by~\cite{CHH}*{Proposition~4.1}.
\end{example}

For the rest of this section, we restrict ourselves to describing properties of $\ucp$, since analogous results hold for $\ucpbar$ after changing orientations. 

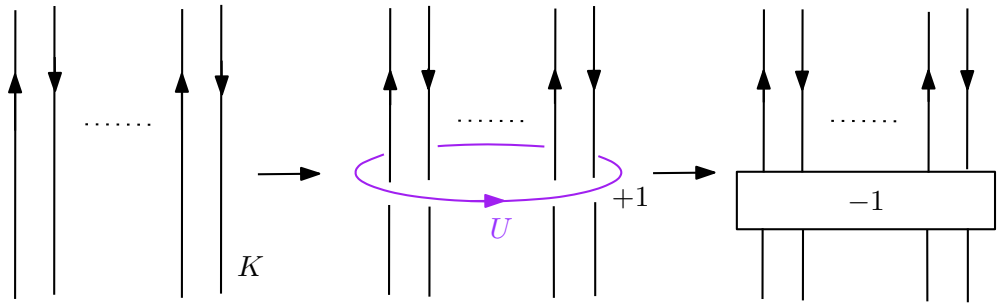
\begin{figure}[htb]
\centering
\begin{tikzpicture}
\node[anchor=south west,inner sep=0] at (0,0){\includegraphics[width=13cm]{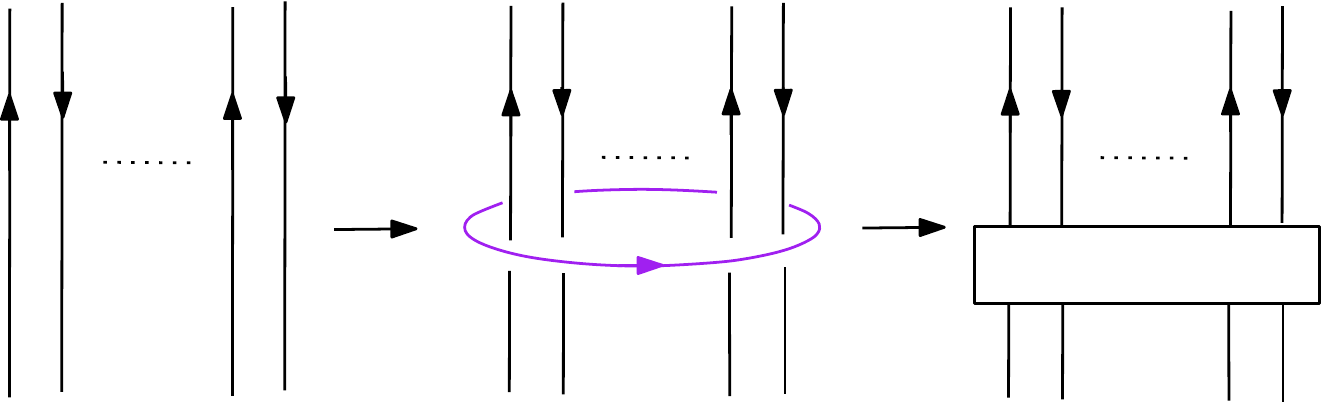}};
\node at (3.2,0.5) {$K$};
\node at (8.2,1.4) {$+1$};
\node at (11.3,1.35) {$-1$};
\node at (6.5,1) {\pink{$U$}};
\end{tikzpicture}
\caption{Left: We begin with an algebraically zero collection of strands in the diagram for a knot $K\subseteq S^3$. Middle: Perform $+1$-framed Dehn surgery on an unknot $U$ in $S^3$, encircling the given strands. Right: In the resulting copy of $S^3$, the strands gain a full negative twist. Passing from left to right in this figure consists of adding a generalized positive crossing.}\label{fig:gencrossing}
\end{figure}

\begin{definition}[\cite{CochranTweedy}*{Definition~2.7}]
\emph{Adding a generalized positive crossing} to a knot $K$ consists of the move shown in \cref{fig:gencrossing}. Specifically, a full positive twist is added to an algebraically zero collection of strands in a diagram for $K$. Equivalently, one finds an unknotted, nullhomologous circle $U\subseteq S^3\setminus K$ and performs $+1$ surgery on $S^3$ along~$U$. 

There is also the analogous notion of adding a generalized negative crossing to a knot $K$.
\end{definition}

The minimum number of generalized crossings one must add to the unknot to construct a given knot is called its \emph{untwisting number}, and was first defined in~\cite{mathieu}. The untwisting number of knots was investigated in~\cites{livingston-slicing,CochranTweedy,KI,KI2,mccoy-nullhomologous,mccoy-gaps,livingston-nullhomologous}. It is straightforward to see that if a knot is obtained from a slice knot by adding $m$ positive generalized crossings, then it is slice in $\#^m \CP^2$~\cite{CochranTweedy}*{Lemma~2.8}. The following result of Cochran and Tweedy 
determines the extent to which the converse holds. 

\begin{lemma}[\cite{CochranTweedy}*{Theorem~5.7}]\label{lem:slicetotwist}
$K$ is $H$-slice in $\#^m \CP^2$ for some $m\geq 0$ if and only if $K$ is concordant to some knot $K'$ that is obtained from a ribbon knot $R$ by inserting $m$ generalized positive crossings. 
\end{lemma}

\begin{remark}\label{rem:negging}
Note that if a knot $K$ is obtained from a knot $R$ by adding generalized positive crossings, then $R$ is obtained from $K$ by adding generalized negative crossings. In particular, if $R$ is topologically slice and $R$ is obtained from $K$ by adding $m$ generalized negative crossings, then $K$ is topologically $H$-slice in $\#^m \CP^2$. 
\end{remark}

For a knot $K$,  let $\sigma_K \colon S^1 \to \Z$ denote the Levine-Tristram signature function. The values of the latter provide effective lower bounds on $\ucp^\top$ by the following result~\cites{murasugi-signature,levine-signature, tristram,kaufman-taylor-signature,cappell-shaneson-signature,viro-signature,gilmer-signature}. 

\begin{proposition}
\label{prop:CNbound}
If $K$ is topologically $H$-slice in  $\#^m \CP^2$,  then 
for any $\omega \in S^1$ with $\Delta_K(\omega) \neq 0$, 
\[-2m \leq \sigma_K(\omega) \leq 0.\]
\end{proposition} 
The most general form of the above statement, for immersed null-homologous concordances between colored links in a general $4$-manifold is given in~\cite{CN}*{Theorem 3.8}. 

Casson-Gordon signatures should also give lower bounds for $\ucp$ and  $\ucpbar$,  via a direct adaptation of the proof of~\cite{CN}*{Theorem~4.12} to apply to $H$-sliceness in general simply connected 4-manifolds. (We remark that the adjective `even' will be absent from the statement of such a generalization, since $\#^m \CP^2$ is not spin.) We expect that such a lower bound could be used to give an alternate proof of the second statement of Theorem~\ref{thm:finite-and-big}.

\begin{example}\label{ex:CNbound-2bridge}
 Let $K$ denote the connected sum of $m$ copies of the right-handed trefoil, for $m\geq 0$. Then $K$ can be unknotted by $m$ positive to negative crossing changes, so $\ucp(K)\leq m$. On the other hand, $\sigma(K)=-2m$, so indeed $\ucp(K)=m$. The identical argument also applies to the torus knot $T_{2,2m+1}$. 

Similarly, consider the $2$-bridge knot $J_m$, for $m\geq 0$, shown in \cref{fig:CNbound-2bridge}. We know that $\ucp(J_m)\leq m$ for $m\geq 0$, since $J_m$ can be changed to the stevedore knot $J_0$ (which is slice) by $m$ positive to negative crossing changes. On the other hand, using the Goeritz matrix one can compute that $\sigma(J_m)=-2m$, and so $\ucp(J_m)=m$ by~\cite{CN}*{Theorem~3.8}.
\begin{figure}[htb]
\centering
\begin{tikzpicture}
\node[anchor=south west,inner sep=0] at (0,0){\includegraphics[width=7cm]{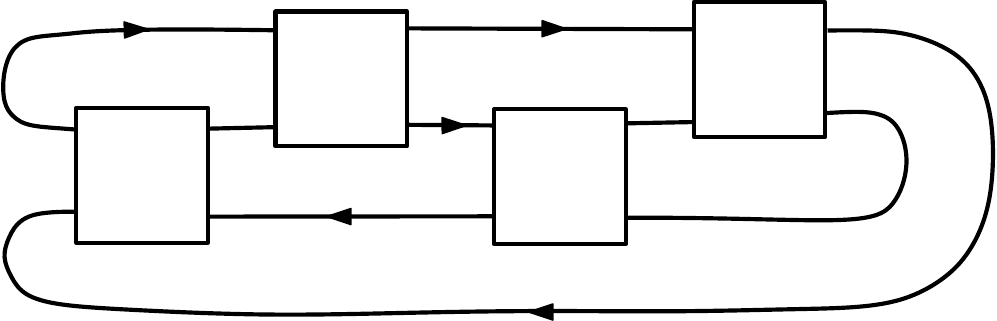}};
\node at (1,1) {$-3$};
\node at (2.4,1.7) {$2m$};
\node at (3.9,1) {$-1$};
\node at (5.3,1.7) {$2$};
\end{tikzpicture}
\caption{The knot $J_m$ from \cref{ex:CNbound-2bridge}. Each box indicates the number of positive half-twists.}
\label{fig:CNbound-2bridge}
\end{figure}
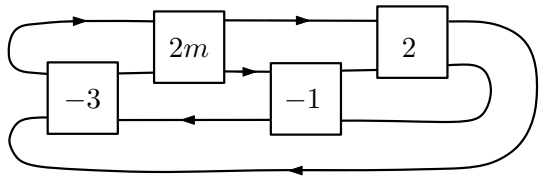

All of the above knots have negative signature and therefore infinite $\ucpbar$ (cf. the knots in \cref{thm:finite-and-big}, for which $\ucp$ and $\ucpbar$ are both large and finite). 
\end{example}

We now return to the remark made in \cref{sec:introduction}, that the $\tau$ invariant can sometimes detect that the unknotting numbers $\ucp$ or $\ucpbar$ are infinite but that it otherwise does not in general help compute their value. In the example used there, we observed that for every $p\geq 1$ there exists a knot with the property that $\tau(K_p)=p$ and $\ucp(K_p)=1$.  Let $K_p$ denote the $(p, 1)$ cable of the right-handed trefoil. Indeed, $\tau(K_p)=p$, and hence $\ucp(K_p)>0$.  
The latter claim follows from the fact that the right-handed trefoil is $H$-slice in $\CP^2$ and the following result.

\begin{proposition}\label{ex:satellite-slice}
Given $K\subseteq S^3$ and a pattern $P\subseteq S^1\times D^2$ with $P(U)$ slice in $B^4$, if $K$ is ($H$-)slice in a $4$-manifold $W$, then so is $P(K)$.
\end{proposition}
\begin{proof}
Since $K$ is slice in $W$ via a disk $\Delta$, there is an embedded annulus $S^1\times [0,1]\subseteq W\setminus (\mathring{B}^4\sqcup \mathring{B}^4)$ cobounded by $K\subseteq S^3$ and the unknot. Use an embedding $P\times [0,1]\subseteq A\times D^2\cong S^1\times D^2\times [0,1]$ to get an annulus in $W\setminus (\mathring{B}^4\sqcup \mathring{B}^4)$ cobounded by $P(K)\subseteq S^3$ and $P(U)$. Since $P(U)$ is slice in $B^4$, we can cap off one end of the annulus to get a disk bounded by $P(K)$ in $W\setminus \mathring{B}^4$. A Mayer-Vietoris argument shows that the disk is trivial in homology, if $\Delta$ were null-homologous at the start of the proof.
\end{proof}

A similar proof as above shows that if $K$ is slice in $W$ (but not necessarily $H$-slice) and $Q\subseteq S^1\times D^2$ is a winding number $0$ pattern with $Q(U)$ slice in $B^4$, then $Q(K)$ is $H$-slice in $W$.

\subsection{Branched covers and plumbings}\label{sec:branched-covers}

\begin{notation}
Let $Y$ be a codimension 2 properly embedded submanifold of $X$, equipped with a map $H_1(X \smallsetminus Y) \to \Z/2$. We denote the double branched cover of $X$ along $Y$ induced by this map by $\Sigma_2(X, Y)$. Occasionally we write simply $\Sigma_2(Y)$ or $\Sigma_2(X)$, when the other argument is implicit. 

We will only need to consider the following three cases: 
\begin{itemize}
\item $X=S^3$ and $Y$ a knot $K\subseteq S^3$, in which case we write $\Sigma_2(K)$;
\item $X$ is a compact, simply connected, oriented $4$-manifold with connected boundary and $Y$ is a properly embedded, oriented, null-homologous disk;
\item $X$ is a compact, simply connected, oriented $4$-manifold whose boundary has two connected components and $Y$ is a properly embedded, oriented, null-homologous annulus between knots in the two boundary components. 
\end{itemize}
In each of these cases, it is straightforward to see from the Mayer-Vietoris sequence that there is a canonical isomorphism $H_1(X\smallsetminus Y)\cong \Z$, where the generator is given by a positive meridian (of the knot, disk, or annulus, respectively), and we will use the map to $\Z/2$ obtained by composing with the quotient map $\Z\to \Z/2$ .
\end{notation}

In \cref{sec:alt,sec:pretzels} we will need explicit representations of the $2$-fold branched covers of $S^3$ branched along twist knots (as well as their connected sums) and pretzel knots. We briefly indicate now how these can be obtained. 

\begin{example}\label{ex:2-bridge-lens}
Seifert observed that for $K$ a $2$-bridge link the double branched cover $\Sigma_2(K)$ is a lens space~\cite{schubert-2bridge}*{Satz 6}. Therefore the classification of lens spaces~\cite{reidemeister} (see also~\cite{brody}) informs the classification of $2$-bridge links, and it was shown by Schubert in~\cite{schubert-2bridge} that the correspondence between $2$-bridge links and lens spaces is a bijection up to certain links with linking number zero (see~\cite{burde-zieschang-book}*{Theorem~12.6, Remark~12.7}). Finally it was shown by Hodgson and Rubinstein that every lens space is the double branched cover for a unique link in $S^3$, which must in particular be a $2$-bridge link~\cite{hodgson-rubinstein}*{Corollary~4.12}. To summarize, given coprime integers $p>q>0$ and a continued fraction expansion
\[
\frac{p}{q}=c_1+\cfrac{1}{c_2+\cfrac{1}{\ddots+\cfrac{1}{c_n}}} =:[c_1,c_2,\dots,c_n]^+
\]  
with $c_1,c_2,\dots,c_n \in \Z$, the lens space $L(p,q)$ is (orientation preserving) homeomorphic to $\Sigma_2(K)$ for $K$ the $2$-bridge link corresponding to the list $[c_1,c_2,\dots, c_n]$ in Conway's $4$-plat notation~\citelist{\cite{conway-4plat}\cite{rolfsen}*{p.\ 303}}. In particular, for the twist knot $K_a$ depicted in~\cref{fig:2bridge-diag}, the double branched cover $\Sigma_2(K_a)$ is the lens space $L(4a+1,2)$, corresponding to the continued fraction $2a+\frac{1}{2}$.\footnote{As always, one must be careful with orientation conventions. Ours matches that of~\cite{lisca} and disagrees with that of~\cite{hodgson-rubinstein} (see~\cite{hodgson-rubinstein}*{Remark~4.13}).} See~\cite{burde-zieschang-book}*{Chapter~12} for further details about $2$-bridge knots and links.
\end{example}

Given $p>q>0$ coprime integers, the lens space $L(p,q)$ is the boundary of both a positive definite and a negative definite $4$-manifold. Specifically, if we have a continued fraction expansion
\begin{equation}\label{eq:negative-cont-frac}
\frac{p}{q}=a_1-\cfrac{1}{a_2-\cfrac{1}{\ddots-\cfrac{1}{a_n}}}=:[a_1,a_2,\dots,a_n]^-
\end{equation}
with integers $a_1,a_2, \dots, a_n\geq 2$ then $L(p,q)$ is the boundary of the plumbed $4$-manifold $W(p,q)$ shown in~\cref{fig:plumbing-background} (see e.g.\ \cite{rolfsen}*{p.\ 272}), whose intersection form is easily seen 
to be negative definite, e.g.\ using Sylvester's criterion. To get a positive definite filling for $L(p,q)$, find a negative definite filling for $-L(p,q)$, which is (orientation preserving) homeomorphic to $L(p,p-q)$, and then switch orientations. 

\begin{figure}[htb]
\begin{tikzpicture}
\node[anchor=south west,inner sep=0] at (0,0){\includegraphics[width=5cm]{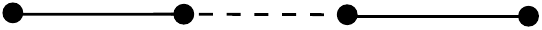}};
\node at (0,-0.2) {$-a_1$};
\node at (1.7,-0.2) {$-a_2$};
\node at (3.2,-0.2) {$-a_{n-1}$};
\node at (4.9,-0.2) {$-a_n$};
\end{tikzpicture}
\caption{A linear plumbing diagram for the $4$-manifold $W(p,q)$ with boundary $L(p,q)$. Here, $\frac{p}{q}$ has the continued fraction expansion shown in~\eqref{eq:negative-cont-frac}, with $a_1,a_2,\dots,a_n\geq 2$. Adjacent generators have intersection $+1$.  }\label{fig:plumbing-background}
\end{figure}
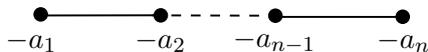

\begin{convention}
A weighted tree $G$, e.g.\ in~\cref{fig:plumbing-background}, produces a Kirby diagram for a $4$-manifold $W(G)$ by the convention that each vertex $v$ with weight $w$ corresponds to an unknot $U(v)$ with framing $w$, where $U(v)$ and $U(v')$ are clasped together whenever there is an edge joining $v$ and $v'$. For example, the diagram in \cref{fig:short-stick} gives a Hopf link where the components have framing $-2a_i-1$ and $-2$. The intersection form on the $4$-manifold given by the weighted graph is represented by its adjacency matrix, where the diagonal entries are the vertex weights (see~\cite{GS}*{Section~6.1} for further details). Specifying a matrix involves choosing a basis for $H_2(W(G);\Z)$. As usual, for each $v$ we get a basis element by taking the union of a disk bounded by $U(v)\subseteq S^3$ and the core of the attached $2$-handle, and we may choose orientations so that the edges in $G$ correspond to intersection $+1$, or alternatively to intersection $-1$. We will use the first convention in the proof of \cref{prop:two-bounds,prop:equations} and the second in the proof of \cref{prop:connectedsumsoftwist}. 
\end{convention}

\begin{example}\label{ex:pretzel-cover}
Given a pretzel knot $K:=P(q_1,q_2,\dots,q_{p+n})$, for integers $p,n\geq 0$, the double branched cover $\Sigma_2(K)$ is a Seifert fibered space with $p+n$ singular fibers~\cite{montesinos}. If $\tfrac{1}{q_1}+\frac{1}{q_2}+\cdots+\frac{1}{q_{p+n}}>0$ then $\Sigma_2(K)$ bounds a plumbed $4$-manifold with negative definite intersection form~\cite{neumann-raymond}*{Theorem~5.2}. The plumbing graph is called the \emph{canonical negative plumbing tree}, when it satisfies a certain list of conditions~\cite{neumann}; in particular the weights of vertices with valence $\leq 2$ must be $\leq -2$. In the case where $p=n+1$ of the parameters of $K$ are odd, positive, and $\geq 3$, and the remaining $n$ parameters are odd, negative and $\leq -3$, the canonical negative definite plumbing tree has the form shown in \cref{fig:pretzel-plumbing}.
Note in particular that the oriented homeomorphism type of $\Sigma_2(K)$ does not depend on the ordering of the parameters~$\{q_i\}$. 
\end{example}
 
\subsection{Integral lattices}
\begin{definition}
An \emph{integral lattice} is a pair $(G,Q)$ where $G$ is a finitely generated free abelian group and $Q\colon G\times G\to \Z$ is a symmetric bilinear form.
\end{definition}

We will consider the integral lattice determined by the intersection form on the absolute second homology (modulo torsion) with integer coefficients of a $4$-manifold.  We remark that when the 4-manifold has boundary,  this intersection form will generally not be unimodular; but in all of our examples the boundary will be a rational homology sphere and so it will at least be nondegenerate. 

We will show in \cref{thm:beast} that if a knot $K$ is positively slice, then the double branched cover of $K$ must bound a $4$-manifold whose intersection lattice is of the following type.

\begin{definition} \label{def:half-int}
A nondegenerate symmetric bilinear pairing $Q \colon \Z^{2m} \times \Z^{2m} \to \Z$ is said to be of \emph{(positive) half-integer surgery type} if there exists an ordered basis $u_1, \dots, u_m, v_1, \dots v_m$ for $\Z^{2m}$ such that the matrix of $Q$ with respect to this basis is of the form
\[
\left[
\begin{array}{cc}
2I_{m \times m} & I_{m \times m}\\
I_{m \times m} & A_{m \times m}
\end{array}
\right]
\]
for some integer matrix $A$, where $I_{m\times m}$ denotes the $m\times m$ identity matrix. 

The pairing $Q$ is said to be of negative half-integer surgery type if the above holds but with $2I_{m\times m}$ replaced by $-2I_{m\times m}$. 
\end{definition}

In the proof of \cref{thm:beast}, we will use the following results from~\cite{Owensslicing} concerning half-integer surgery type lattices.

\begin{lemma}[\cite{Owensslicing}*{Lemma~2.2}]\label{lem:mod-2-enough}
Let $Q$ be a $2m \times 2m$ integer block matrix consisting of $m\times m$ blocks of the form 
\[Q=
\left[
\begin{array}{cc}
2I_{m \times m} & *\\
* & *
\end{array}
\right]
\]
which is congruent modulo $2$ to the matrix $
\left[
\begin{array}{cc}
2I_{m \times m} & I_{m \times m}\\
I_{m \times m} & A_{m \times m}
\end{array}
\right]
$
for some integer matrix $ A_{m \times m}
$. Then there exists a matrix $P=\left[
\begin{array}{cc}
I_{m \times m} & *\\
0_{m\times m} & R_{m \times m}
\end{array}
\right]\in \GL(2m, \Z)$
so that 
\[
P^TQP=\left[
\begin{array}{cc}
2I_{m \times m} & I_{m \times m}\\
I_{m \times m} & A'_{m \times m}
\end{array}
\right],
\]
where $A'\equiv A \mod{2}$. 
\end{lemma}
In a similar vein, the following result gives another convenient means of detecting an integer lattice of half-integer surgery type.

\begin{proposition}[\cite{Owensslicing}*{Proposition~2.4}]\label{prop:finite-enough}
Suppose $Q$ is a positive definite integer lattice and
$L$ is a sublattice of $Q$ of odd index. If $L$ is of half-integer surgery type, then so is $Q$.
\end{proposition}

\section{\texorpdfstring{Proof of \cref{thm:beast}}{Proof of Theorem~1.2}}\label{sec:beast}

We restate \cref{thm:beast} and then give the proof.

\begin{reptheorem}{thm:beast}
Let $K\subseteq S^3$ be a knot with $\sigma(K)=0$. Suppose  $K$ is $H$-slice in $\#^m\CP^2$ for some $0\leq m< \infty$. Then $\Sigma_2(K)$, the double
cover of $S^3$ branched along $K$, bounds a compact, smooth, oriented $4$-manifold $X$ with $b_2(X) =2m$, whose intersection form is positive definite and of half-integer surgery type.
\end{reptheorem}

\begin{proof}
Since $K$ is $H$-slice in $\#^m\CP^2$, there exist knots $K', R\subseteq S^3$ as in \cref{lem:slicetotwist}, so that $R$ is ribbon, $K$ is concordant to $K'$, and $K'$ is obtained from $R$ by inserting $m$ positive generalized crossings. In other words, there is a link diagram $R\sqcup \gamma_1\sqcup \dots \sqcup \gamma_m\subseteq S^3$ such that the curves $\gamma_1, \dots, \gamma_m$ appear as the standard unlink, each $\gamma_i$ has trivial linking number with $R$, and performing $+1$-framed Dehn surgery on each component of $\gamma_1, \dots, \gamma_m$ transforms $R$ into~$K'$ (see Figure \ref{InsTw}).

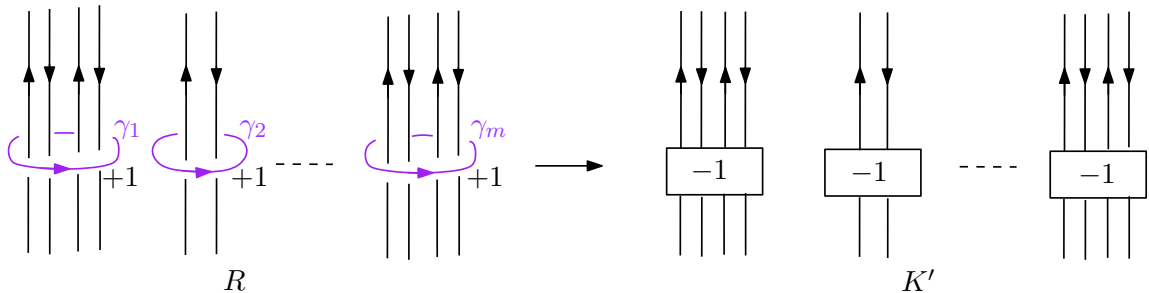
\begin{figure}[htb]
\centering
\begin{tikzpicture}
\node[anchor=south west,inner sep=0] at (0,0){\includegraphics[width=15cm]{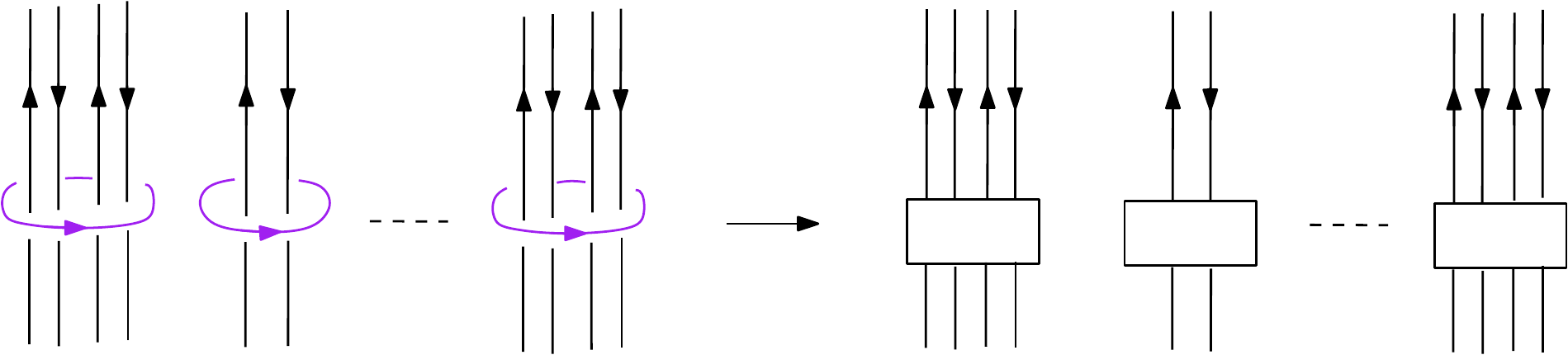}};
\node at (3,-0.25) {$R$};
\node at (1.6,1.75) {\pink{$\gamma_1$}};
\node at (3.25,1.75) {\pink{$\gamma_2$}};
\node at (6.35,1.75) {\pink{$\gamma_m$}};
\node at (1.5,1.1) {$+1$};
\node at (3.2,1.1) {$+1$};
\node at (6.3,1.1) {$+1$};
\node at (9.25,1.2) {$-1$};
\node at (11.35,1.17) {$-1$};
\node at (14.35,1.15) {$-1$};
\node at (12,-0.25) {$K'$};
\end{tikzpicture}
\caption{Proof of~\cref{thm:beast}. The knots $R$ and $K'$ are shown in black, and the link $\gamma_1\sqcup \gamma_2 \sqcup \cdots \sqcup \gamma_m$ in purple. We perform $+1$-framed Dehn surgery on $S^3$ along each $\gamma_i$, or equivalently, add generalized positive crossings to $R$ guided by $\gamma_1\sqcup \gamma_2 \sqcup \cdots \sqcup \gamma_m$. Compare with \cref{fig:gencrossing}.} \label{fig:schematic-beast-setup}\label{InsTw}
\end{figure}
\begin{figure}[htb]
\centering
\begin{tikzpicture}
\node[anchor=south west,inner sep=0] at (0,0){\includegraphics[width=11.5cm]{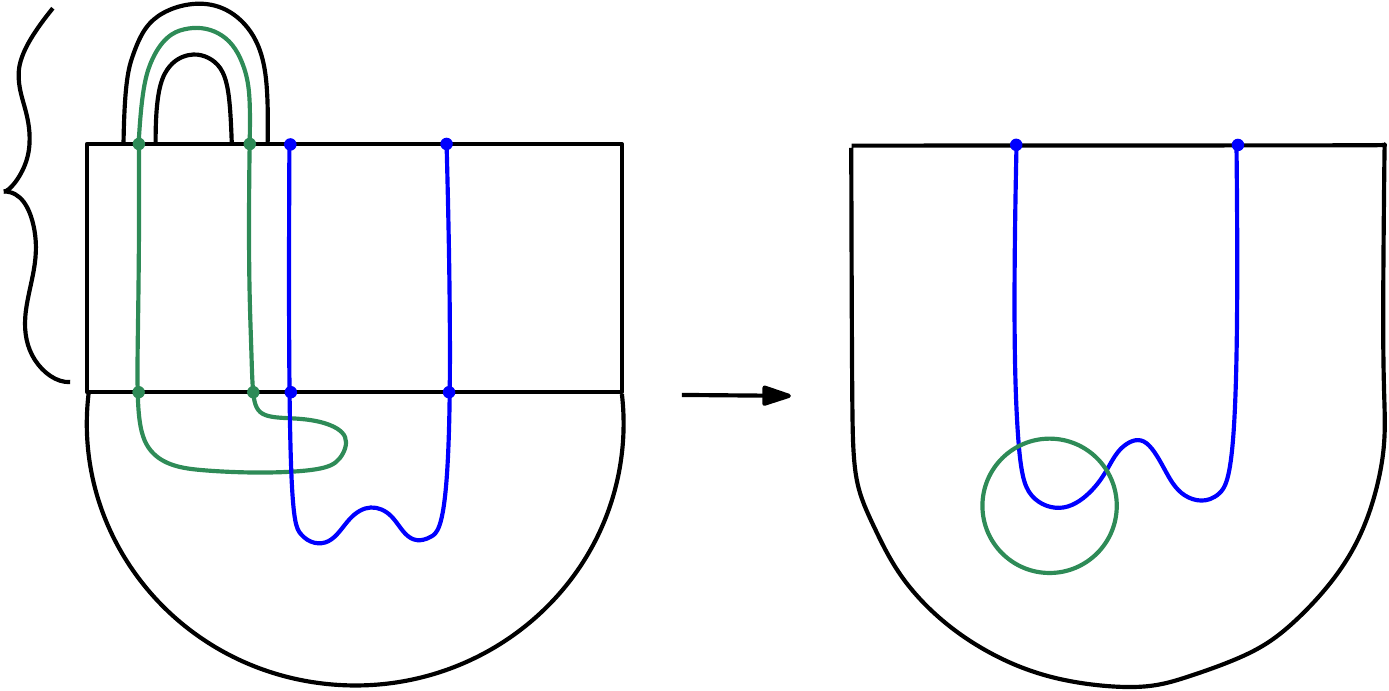}};
\node at (-0.25,4.1) {$W$};
\node at (2.6,1.65) {$+$};
\node at (2.6,2.35) {$-$};
\node at (8.25,2.15) {$+$};
\node at (9.4,1.85) {$-$};
\node at (1.8,2.75) {\green{$\gamma_1$}};
\node at (3,6) {$+1$-framed $2$-handle};
\node at (4,1.5) {\blue{$D_R$}};
\node[rotate=90] at (3.45,3.5) {\blue{$R\times [0,1]$}};
\node[rotate=90] at (4.85,3.5) {$S^3\times [0,1]$};
\node at (5.2,1) {$B^4$};
\node at (10.7,3.8) {\blue{$D_{K'}$}};
\node at (10.5,4.8) {\blue{$K'$}};
\node at (12.2,2.25) {$(\CP^2)^{\times}$};
\node at (6,2.7) {$\boldsymbol{\cong}$};
\end{tikzpicture}
\caption{Proof of~\cref{thm:beast} continued. Left: The ribbon disk $D_R$ for $R$ is glued on to $R\times [0,1]$ to produce a disk within $W\cup B^4$. Here $W$ is obtained from $S^3\times [0,1]$ by attaching $+1$-framed $2$-handles along $\gamma_i\times \{1\}\subseteq S^3\times \{1\}$, for each $i=1, \dots, m$. The case $m=1$ is pictured. Depicted in green is a $+1$-framed $2$-sphere obtained as the union of the core of the $2$-handle attached to $\gamma_1\times \{1\}$, the cylinder $\gamma_1\times [0,1]$, and a slice disk for $\gamma_1$ in $B^4$. Right: The slice disk $D_{K'}$ for $K'$ in $(\#^m \CP^2)^{\times}$ is shown. It is glued on to the concordance $C\subseteq S^3\times [0,1]$ from $K'$ to $K$ to produce a slice disk for $K$ in $(\#^m \CP^2)^{\times}$, The case $m=1$ is pictured, and the core $\CP^1\subseteq (\CP^2)^\times$ is shown in green. As indicated there is a diffeomorphism taking the picture on the left to the picture on the right.}\label{fig:beast1}
\end{figure}
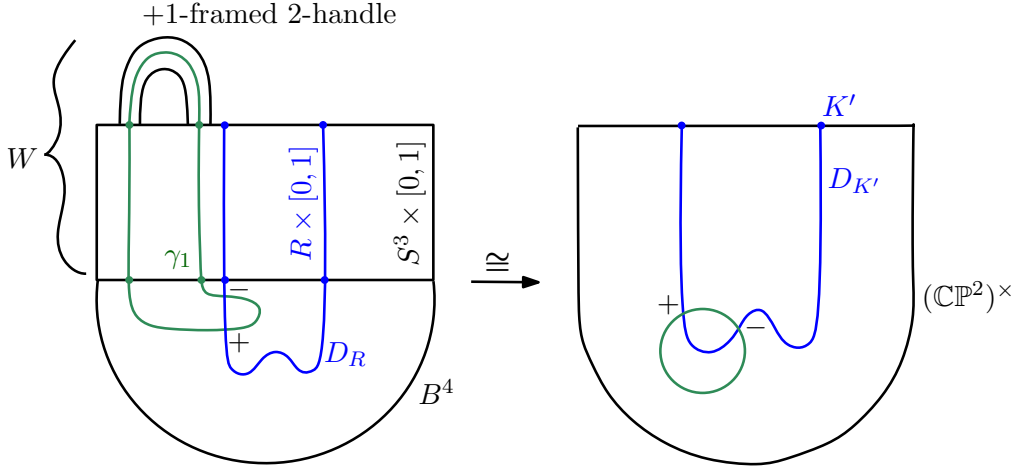 

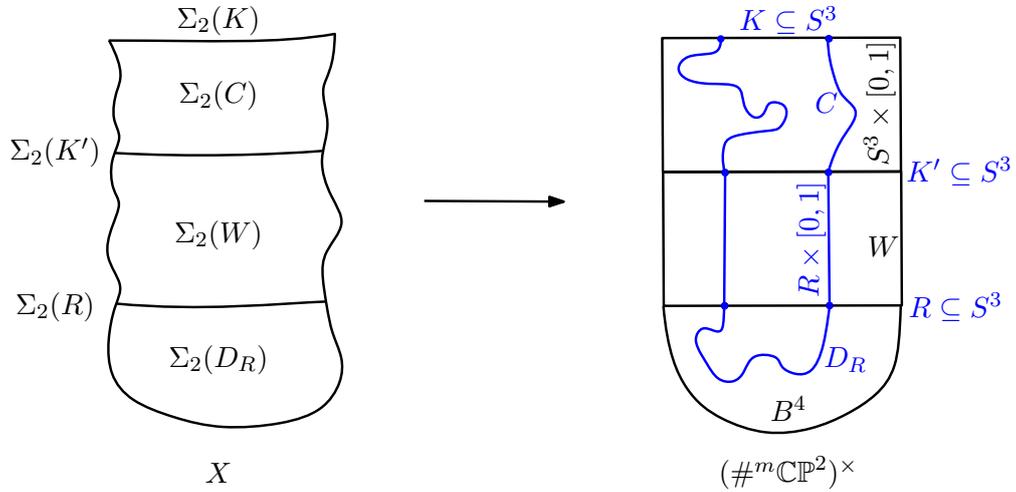
\begin{figure}[htb]
\centering
\begin{tikzpicture}
\node[anchor=south west,inner sep=0] at (0,0){\includegraphics[width=10.5cm]{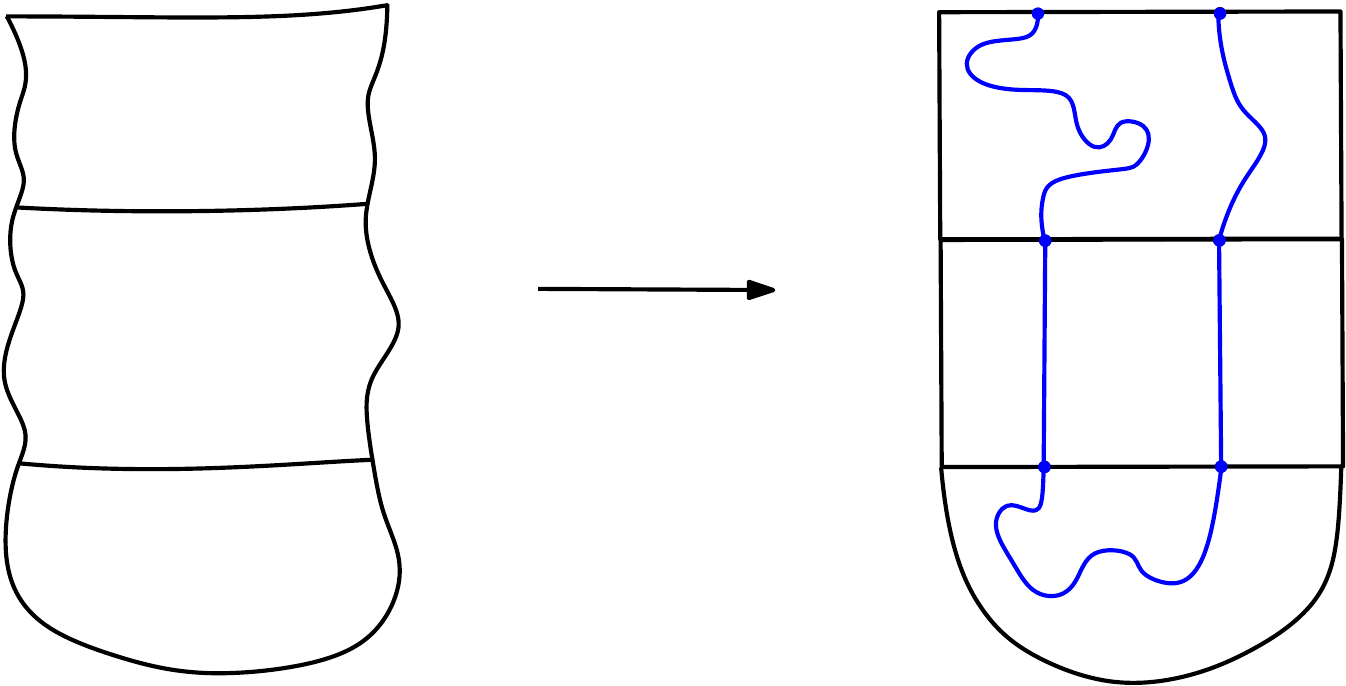}};
\node at (1.5,1) {$\Sigma_2(D_R)$};
\node at (1.5,2.65) {$\Sigma_2(W)$};
\node at (1.5,4.5) {$\Sigma_2(C)$};

\node at (-0.65, 1.7) {$\Sigma_2(R)$};
\node at (-0.65,3.75) {$\Sigma_2(K')$};
\node at (1.5,5.5) {$\Sigma_2(K)$};

\node at (9,0.35) {$B^4$};
\node at (10.25,2.5) {$W$};
\node[rotate=90] at (10.2,4.4) {$S^3\times [0,1]$};

\node at (11.2, 1.7) {\blue{$R\subseteq S^3$}};
\node at (11.25,3.5) {\blue{$K'\subseteq S^3$}};
\node at (9,5.5) {\blue{$K\subseteq S^3$}};

\node at (9.75,1) {\blue{$D_R$}};
\node[rotate=90] at (9.3,2.6) {\blue{$R\times [0,1]$}};
\node at (9.5,4.4) {\blue{$C$}};

\node at (1.5,-0.5) {$X$};
\node at (9,-0.5) {$(\#^m \CP^2)^\times$};
\end{tikzpicture}
\caption{Proof of~\cref{thm:beast} continued. The double branched cover $X$ of punctured $\#^m \CP^2$ along a slice disk for $K$.}
\label{fig:beastbig}
\end{figure}

Let $D_R$ be a ribbon disk for $R$ in $B^4$ and regard the link  $R\sqcup \gamma_1\sqcup \dots \sqcup \gamma_m$ as living in $S^3\times\{1\}\subseteq S^3\times [0, 1]$. Let $W$ be equal the union of $S^3 \times [0,1]$ and $m$ copies of the standard $2$-handle, attached along $\gamma_1, \dots, \gamma_m\subseteq S^3\times \{1\}$, each with framing $+1$. By construction, $W$ is diffeomorphic to twice-punctured $\#^m \CP^2$. 
Consider
 \[ D_R \cup_{R\times\{0\}} (R\times [0,1])\subseteq B^4 \cup_{\partial_- W} W \cong (\#^m \CP^2)\smallsetminus \mathring{B}^4=:(\#^m \CP^2)^{\times},\]
 and observe that the image of $R \times \{1\}$ in $\partial(\#^m \CP^2)^{\times}$ is the knot $K'$ in $S^3$. 
 In particular, the disk $$D_{K'}:=D_R \cup (R\times [0,1])$$
is a slice disk for $K'$ in $(\#^m \CP^2)^{\times}$. See \cref{fig:beast1}. 

Consider the double branched covers:
\begin{align*}
\Sigma_2(D_R):=&\Sigma_2(B^4, D_R),\\
\Sigma_2(W):=&\Sigma_2(W, R\times [0,1]),\\
\Sigma_2(D_{K'}):=&\Sigma_2((\#^m \CP^2)^{\times},D_{K'}) \cong \Sigma_2(W)\cup_{\Sigma_2(R)} \Sigma_2(D_R),\\
\Sigma_2(C):=&\Sigma_2(S^3 \times [0,1], C),
\end{align*}
where $C$ denotes the concordance between $K$ and $K'$ guaranteed by \cref{lem:slicetotwist} (see~\cref{fig:beastbig}). The diffeomorphism $\Sigma_2(D_{K'}) \cong \Sigma_2(W)\cup_{\Sigma_2(R)} \Sigma_2(D_R)$ is induced by the diffeomorphism $ B^4 \cup_{\partial_- W} W \cong (\#^m \CP^2)^{\times}$. Let $X$ be the result of gluing $\Sigma_2(D_{K'})$ to $\Sigma_2(C)$, along $\Sigma_2(K')$. Then, by construction, $X$ is obtained as a double cover of $(\#^m \CP^2)^{\times}$ branched along a null-homologous slice disk for $K$. In particular,  $\partial X=\Sigma_2(K)$. 

Apply the proof of~\cite{Cochran-Lickorish}*{Theorem~3.7} with $p_-=n_+=n_-=0$ and $p_+=m$
to see that $\Sigma_2(D_{K'})$ is a smooth, oriented $4$-manifold with positive definite intersection form and $b_2=2m$; this uses that $\sigma(K)=0$. 

Similarly, we know that $\Sigma_2(C)$ is a $\Z/2$-homology $S^3 \times [0,1]$. 
Then $X=\Sigma_2(D_{K'})\cup \Sigma_2(C)$ is a smooth, oriented $4$-manifold and the intersection form $Q_X\colon H_2(X) \times H_2(X) \to \Z$
 is positive definite. 
We need to show that $Q_X$ is of half-integer surgery type.  Since $X$ is obtained by gluing together $\Sigma_2(C)$, and manifolds diffeomorphic to $\Sigma_2(W)$ and $\Sigma_2(D_R)$, along rational homology spheres (namely, the double branched covers $\Sigma_2(K')$ and $\Sigma_2(R)$), the intersection form on $\Sigma_2(W)$ is an odd index sublattice of the intersection form on $X$. This uses that $\Sigma_2(D_R)$ is a $\Z/2$-homology $B^4$ (see, for example,~\cite{CG86}*{Lemma~2}).

We will next show that the intersection form on $\Sigma_2(W)$ is of half-integer surgery type. By \cref{prop:finite-enough}, this will suffice to show that the intersection form on $X$ is also of half-integer surgery type, as desired. 

\begin{claim}
$Q_{\Sigma_2(W)}$ is of half-integer surgery type.
\end{claim}

\begin{proof}[Proof of claim]
Our goal is to replicate the argument of~\cite{Owens08}*{Lemma~3.1}; the difficulty in doing so is that, since $R$ is not the unknot, $\Sigma_2(R)$ is only known to be a $\Z/2$-homology sphere, rather than $S^3$. This requires extra care with respect to framings. Nonetheless, the argument from~\cite{Owens08} can be generalized to this setting by ensuring we only consider framings of null-homologous loops. 

The boundary of $\Sigma_2(D_R)$ is $\Sigma_2(R)$, in particular, a rational homology sphere. The second homology of $\Sigma_2(W)$, relative to $\Sigma_2(R)$, is generated by the lifts of the $2$-handles, so we investigate what happens to the curves to which they are attached. For the rest of the proof, choose an (arbitrary) orientation on each $\gamma_i$. Since the linking numbers of the curves $\gamma_1, \dots, \gamma_m$ with $R$ are even (in fact, zero), they lift to $2m$ curves in $\Sigma_2(R)$; we denote the two oriented lifts of $\gamma_i$ by $\widetilde{\gamma_i}$ and $\widetilde{\gamma_i}'$. Note that we have no way to distinguish between these two lifts.

In order to attach a $2$-handle along a neighborhood of a curve in a $3$-manifold, we need a parametrization of that neighborhood as $S^1 \times D^2$ or, equivalently, a choice of preferred $S^1 \times *$, where $* \in \partial D^2$.  Let $\alpha_i$ denote a push-off of $\gamma_i$ into the disk region $\Delta_i$ where the generalized positive crossing will be inserted, as shown in~\cref{fig:alpha}. Note that $\alpha_i$ is a push-off in the plane of the diagram and also the $0$-framed longitude of $\gamma_i$. The framing curve for $\gamma_i$, corresponding to the $+1$-framing, is isotopic in $\partial \nu(\gamma_i)$ to $\alpha_i+ \mu(\gamma_i)$, where $\nu(\gamma_i)$ is a small tubular neighbourhood of $\gamma_i$ and $\mu(\gamma_i)$ is a positive meridian. The lift of $\Delta_i$ to $\Sigma_2(R)$ is some compact surface $F_i$ with two boundary components. The Euler characteristic of $F_i$ depends on the \emph{geometric} intersection number of $\Delta_i$ with $R$: in the simplest case of an actual crossing change, it is an annulus. The surface $F_i$ inherits an orientation from $\Delta_i$, and the two (oriented) boundary components of $F_i$ are the two lifts of $\gamma_i$ from before, i.e.\ $\partial F_i=\wt{\gamma}_i\sqcup \wt{\gamma}'_i$, and thus $[\widetilde{\gamma_i}]=-[\widetilde{\gamma_i}']$ in $H_1(\Sigma_2(R);\Z)$. Note that the elements of $\{F_i\}$ are pairwise disjoint and embedded, since the generalized positive crossings are added in isolated balls, separated from one another. 

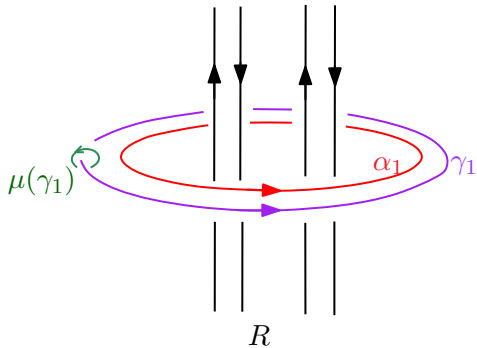
\begin{figure}[htb]
\begin{tikzpicture}
\node[anchor=south west,inner sep=0] at (0,0){\includegraphics[width=5cm]{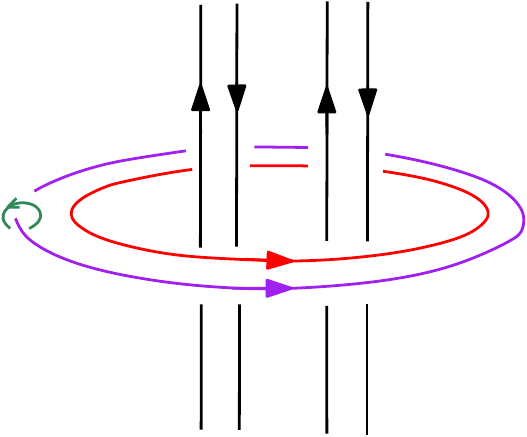}};
\node at (2.5,-0.25) {$R$};
\node at (5.2,2) {\pink{$\gamma_1$}};
\node at (4.2,2) {\red{$\alpha_1$}};
\node at (-0.36,1.8) {\green{$\mu(\gamma_1)$}};
\end{tikzpicture}
\caption{Proof of~\cref{thm:beast} continued. The curves $\gamma_1$, $\alpha_1$, and $\mu(\gamma_1)$ are shown with respect to a ribbon knot $R$. Compare with \cref{fig:schematic-beast-setup}.}
\label{fig:alpha}
\end{figure}

Let $\wt\alpha_i$ and $\wt\alpha_i'$ denote the two lifts of $\alpha_i$ in $\Sigma_2(R)$, corresponding to $\wt\gamma_i$ and $\wt\gamma_i'$ respectively. Note that $\wt\alpha_i$ and $\wt\alpha_i'$ both lie on $F_i$. The framing curve for $\wt\gamma_i$ (induced by the $+1$-framing on $\gamma_i$ in $S^3$)  is isotopic in $\partial \nu(\wt\gamma_i)$ to $\wt\alpha_i+ \mu(\wt\gamma_i)$, where, as before, $\nu(\gamma_i)$ is a small tubular neighbourhood of $\wt\gamma_i$ and $\mu(\wt\gamma_i)$ is a positive meridian. Denote this framing curve by $\eta_i$. Similarly, the framing curve for $\wt\gamma_i'$ is isotopic in $\partial \nu(\wt\gamma_i')$ to $\wt\alpha_i'+ \mu(\wt\gamma_i')$.

For $h_i$ a $2$-handle in the given decomposition of $W$, let $\wt h_i$ and $\wt h_i'$ denote the two lifts to $\Sigma_2(W)$, attached along $\wt\gamma_i$ and $\wt \gamma_i'$ respectively. For each $i$, we will slide $\wt h_i'$ over $\wt h_i$ and call the result $\wt h_i + \wt h_i'$. More precisely, $\wt h_i + \wt h_i'$ is attached along a band sum of $\wt \gamma_i'$ and $\eta_i$, denoted $\zeta_i$, where we require the band to lie on the surface $F_i$ (see Figure \ref{fig:eta}). Note that $\zeta_i$ is null-homologous in $\Sigma_2(R)$ since we know that $[\eta_i]=[\widetilde{\gamma_i}]=-[\widetilde{\gamma_i}']$ in $H_1(\Sigma_2(R);\Z)$. 

Recall that the linking number $\lk_Y(x,y)$ is a well-defined integer for oriented null-homologous loops $x,y$ in a $3$-manifold $Y$ and can be computed as the signed count of intersections between $y$ and any $2$-chain bounded by $x$. As a result, for a fixed such loop $x$, a framing of a tubular neighbourhood also corresponds to a well-defined integer. 
 Moreover, if a $2$-handle $h$ is attached to $Y \times [0,1]$ along $x$ with framing $n$ to form a $4$-manifold $Y(h)$, then $H_2(Y(h),Y\times \{0\};\Z)\cong \Z$ is generated by $h$.  Further, since $x$ is null-homologous in $Y$, the handle $h$ produces a class in $H_2(Y(h);\Z)$ and $Q_{Y(h)}(h,h)=n$. 

Returning to the proof, we have that the framing of $\zeta_i$ equals the value of $Q_{\Sigma_2(W)}(x_i,x_i)$, where $x_i$ is the class in $H_2(\Sigma_2(W);\Z)$ corresponding to the handle $\wt{h}_i+\wt h_i'$. We compute the framing of $\zeta_i$ next. 

Since $\zeta_i$ is obtained from the framed curves $\eta_i$ and $\wt\gamma_i'$ by a band sum along $F_i$, we find that the framing on $\zeta_i$ is $2$. This can be seen more clearly from~\cref{fig:eta}. Specifically, the framing cuve for $\zeta_i$ is obtained from a pushoff of $\zeta_i$ into $F_i$ by summing with $\mu(\wt\gamma_i)$ and $\mu(\wt\gamma_i')$. The framing of $\zeta_i$ is the linking number (in $\Sigma_2(R)$) between $\zeta_i$ and this framing curve. A null-homology for $\zeta_i$ is provided by $F_i$ minus the band used to produce $\zeta_i$ union a meridional disk bounded by $\mu(\wt\gamma_i)$. From the figure we see that there are two points of intersection between this null-homology and the framing curve for $\zeta_i$. Hence $Q_{\Sigma_2(W)}(x_i,x_i)=2$. 

We also see for $i \neq j$ that $\lk_{\Sigma_2(R)}(\zeta_i,\zeta_j)=0$ since the $\{\Delta_i\}$ (and therefore the $\{F_i\}$) are split from one another, and so $Q_{\Sigma_2(W)}(x_i,x_j)=0$ as well.

\begin{figure}[htb]
\begin{tikzpicture}
\node[anchor=south west,inner sep=0] at (0,0){\includegraphics[width=12cm]{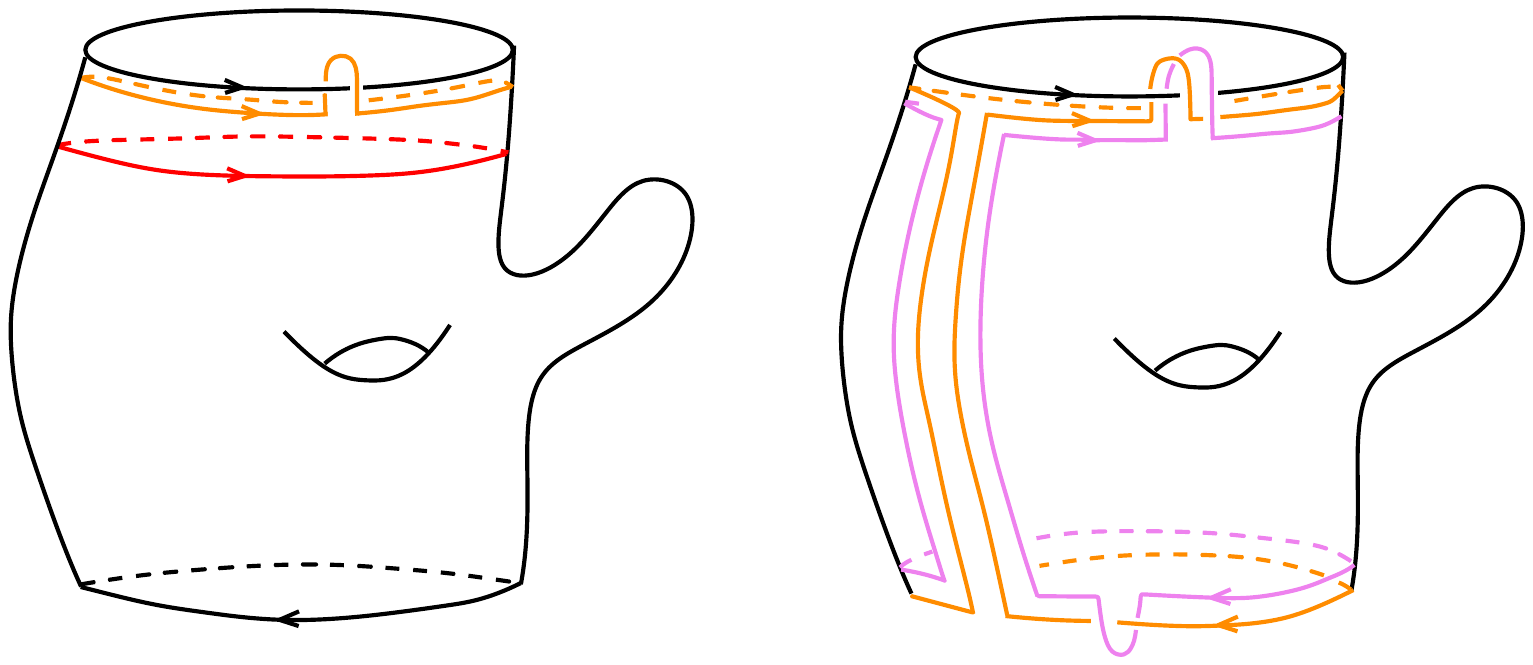}};
\node at (3.8,1.5) {$F_i$};
\node at (4.4,0.5) {$\wt\gamma_i'$};
\node at (4.25,5) {$\wt\gamma_i$};
\node at (4.25,4) {\red{$\wt\alpha_i$}};
\node at (4.25,4.5) {\orange{$\eta_i$}};
\node at (10.8,4.5) {\orange{$\zeta_i$}};
\node at (2.2,-0.5) {(a)};
\node at (9,-0.5) {(b)};
\end{tikzpicture}
\caption{Proof of~\cref{thm:beast} continued. (a) The surface $F_i$ with (oriented) boundary components $\wt\gamma_i$ and $\wt\gamma_i'$. A pushoff $\wt\alpha_i$ (red) of $\wt\gamma_i$ is shown, as well as the framing curve $\eta_i$ (orange) for $\wt\gamma_i$. (b) The curve $\zeta$ (orange) as well as its framing curve (purple) are shown. By construction $\zeta_i$ is the band sum of $\wt\gamma_i'$ and $\eta_i$ along a band lying on $F_i$. }
\label{fig:eta}
\end{figure}

We are still missing half of the desired basis for the intersection form on $\Sigma_2(W)$. The determinant of the ribbon knot $R$, i.e.\ the order of the group $H_1(\Sigma_2(R);\Z)$, is some odd square $k^2$. Then $k^2[\wt\gamma_i]$ is trivial in $H_1(\Sigma_2(R);\Z)$, and we obtain a class $y_i\in H_2(\Sigma_2(W);\Z)$ given by $k^2$ copies of the handle $\wt h_i$ attached along $\wt\gamma_i$, together with the null-homology for $k^2[\wt\gamma_i]$ in $\Sigma_2(R)$. 
 Note that the collection $\{x_i,y_i\}_{i}$ generates a full rank sublattice of $Q_{\Sigma_2(W)}$. Moreover, since $k^2$ is odd, it is an odd index sublattice, as $\{\wt h_i, \wt h_i+\wt h_i'\}$ generate $H_2(\Sigma_2(W),\Sigma_2(R);\Z)$. As before, the values $Q_{\Sigma_2(W)}(x_i, y_j)$ correspond to the linking numbers $\lk_{\Sigma_2(R)}(\zeta_i, k^2[\wt \gamma_j])$. By definition, this is the signed count of intersections between a null homology for $\zeta_i$ and a representative of $k^2[\wt \gamma_j]$. By construction of $\zeta_i$, a null homology consists of the surface $F_i$, minus the band used for the band sum of $\wt\gamma_i'$ and $\eta_i$, union a meridional disk for $\wt\gamma_i$. A representative for $k^2[\wt \gamma_j]$ is provided by the $(k^2,1)$-cable of $\wt\gamma_j$. Then we have that $\lk_{\Sigma_2(R)}(\zeta_i, k^2[\wt \gamma_j])=k^2\delta_{ij}$. Recalling that $k^2$ is odd, we see from~\cref{lem:mod-2-enough} that $Q_{\Sigma_2(W)}$ has an odd index sublattice of half-integer surgery type. By~\cref{prop:finite-enough}, we know that $Q_{\Sigma_2(W)}$ is of half-integer surgery type as well. 
\end{proof}

The proof of the above  claim completes the proof of ~\cref{thm:beast}. \end{proof}

By applying \cref{thm:beast} to the mirror image of a given knot, we obtain a bound for $\ucpbar$, as we now state.  

\setcounter{beastmate}{1}
\begin{beastmate}\label{cor:beastmate}
Let $K\subseteq S^3$ be a knot with $\sigma(K)=0$. Suppose  $K$ is $H$-slice in $\#^m\ol{\CP^2}$ for some $0\leq m< \infty$. Then $\Sigma_2(K)$, the double
cover of $S^3$ branched along $K$, bounds a compact, smooth, oriented $4$-manifold $X$ with $b_2(X) =2m$, whose intersection form is negative definite and of negative half-integer surgery type.
\end{beastmate}

\section{Alternating knots}\label{sec:alt}
The goal of this section is to prove~\cref{thm:finite-and-big}, as an application of~\cref{thm:beast}. As we wish to control both $\ucp$ and $\ucpbar$, we work with $2$-bridge knots, or rather their connected sums. In particular, as mentioned in~\cref{sec:branched-covers}, the double branched cover of $S^3$ branched along any $2$-bridge knot is a lens space, and every lens space is well-known to bound both a positive definite and a negative definite $4$-manifold. More generally, the double branched cover of $S^3$ along any alternating knot has both a positive definite and a negative definite filling, which may be obtained by branching over (pushed-in copies of) the black and white 
checkerboard surfaces 
for the knot
(see, e.g.~\cite{OS-DBC}*{Lemma~3.6}). In our case, we prefer to work with the explicit plumbed fillings for lens spaces and their connected sums described in~\cref{sec:branched-covers}.

We refer the reader to~\cite{lisca} for information on the canonical definite manifolds bounded by the double branched covers of general $2$-bridge knots.  
Here we focus on the family of twist knots $K_a$ indexed by $a \geq 1$, shown in~\cref{fig:2bridge-diag}. In the notation of~\cite{lisca}, the knot $K_a$ corresponds to 
\[\frac{4a+1}{2}= 2a+ \frac{1}{2}=[2a,2]^+.\] 
Casson-Gordon~\cite{CG86} showed that $K_a$ is not slice for $a \geq 3$.  Note that a single negative to positive crossing change transforms $K_a$ into the unknot, while $a-2$ positive to negative crossing changes transform $K_a$ into $K_2$, which is slice. Therefore we have that $\ucp(K_a) \leq a-2$ and $\ucpbar(K_a) \leq 1$ for all $a \geq 3$. As mentioned in~\cref{sec:introduction}, this implies that the Tristram-Levine signature function of $K_a$ is identically zero. Indeed, an infinite family of the knots $\{K_{a}\}$ for $a\geq 3$ are algebraically slice, as shown e.g.\ in~\cite{CG86}*{p.\ 182}.

\begin{figure}[htb]
\begin{tikzpicture}
\node[anchor=south west,inner sep=0] at (0,0){\includegraphics[width=6cm]{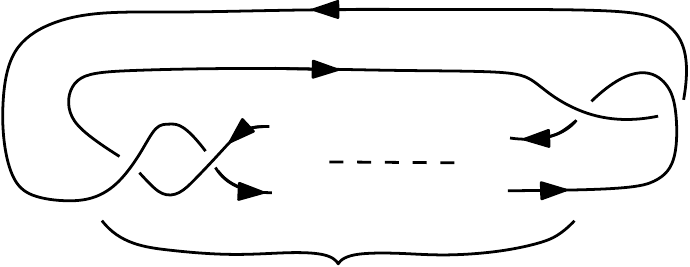}};
\node at (3.4,-0.25) {$2a$ positive crossings};
\end{tikzpicture}
\caption{The twist knot $K_a$. Note that $K_1$ is the figure eight knot and $K_2$ is the stevedore knot.}
\label{fig:2bridge-diag}
\end{figure}

We begin by proving the following result.

\begin{proposition}\label{prop:connectedsumsoftwist}
Let $n \geq 1$ and let $a_i \geq 3$ for all $i$.
Then $\ucpbar(\#_{i=1}^n K_{a_i})=n$. 
\end{proposition}
\begin{proof}
Since $\ucpbar(K_a) \leq 1$ for all $a$, we have the desired upper bound, $\ucpbar(\#_{i=1}^n K_{a_i})\leq n$. 
Suppose that $K:= \#_{i=1}^n K_{a_i}$ is $H$-slice in $\#^m \overline{\CP^2}$. We will show that $m \geq n$. 

We begin by finding a smooth, compact, oriented, positive definite $4$-manifold $W$ such that $\partial W= \Sigma_2(K)$. For any $a$, the knot $K_a$ has double branched cover the lens space $L(4a+1, 2)=-L(4a+1, 4a-1)$.  A short inductive argument shows that $\frac{4a+1}{4a-1}$ has a continued fraction expansion of length $2a$ given by $[2, 2, \dots, 2, 3]^-$.\footnote{Put precisely, let $m_1:=2-\frac{1}{3}=\frac{5}{3}$ and for all $n>1$ recursively define $m_n=2-\frac{1}{m_{n-1}}$. Then $m_n=\frac{2n+3}{2n+1}$.} This implies that $L(4a+1, 4a-1)$ bounds the linear plumbing with framings $-2, -2, \dots, -2, -3$ and hence that $L(4a+1, 2)$ bounds the linear plumbing $W(a)$ with framings $2, 2, \dots, 2, 3$, which is well-known to be (and can be directly verified as) positive definite. Now let $W$ be the boundary connected sum of $\{W(a_1), W(a_2), \dots, W(a_n)\}$. For each $i$, denote the generators of $H_2(W(a_i);\Z)$ by $x_i^1, \dots, x_i^{2a_i-1}, y_i$, as illustrated in \cref{fig:linearplumb}, and oriented so that adjacency corresponds to intersection $-1$, e.g.\ $Q_{W}(x_i^j,x_i^{j+1})=-1$ for all $i$ and for all $j< 2a_i-1$. By construction $\partial W=\Sigma_2(K)$. 

\begin{figure}[htb]
\begin{tikzpicture}
\node[anchor=south west,inner sep=0] at (0,0){\includegraphics[width=5cm]{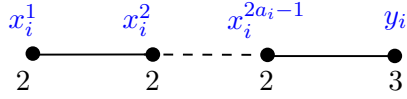}};
\node at (-0,0.6) {\blue{$x_i^1$}};
\node at (1.5,0.6) {\blue{$x_i^2$}};
\node at (3.2,0.6) {\blue{$x_i^{2a_i-1}$}};
\node at (4.9,0.6) {\blue{$y_i$}};
\node at (0,-0.2) {$2$};
\node at (1.7,-0.2) {$2$};
\node at (3.2,-0.2) {$2$};
\node at (4.9,-0.2) {$3$};
\end{tikzpicture}
\caption{A linear plumbing diagram for the 4-manifold $W(a_i)$. Adjacent generators have intersection $-1$. By construction, $\partial W(a_i)=L(4a_i+1,2)=\Sigma_2(K(a_i))$.}\label{fig:linearplumb}
\end{figure}

Since $K$ is $H$-slice in $\#^m \overline{\CP^2}$, we know that $-K$ is $H$-slice in $\#^m \CP^2$, so by \cref{thm:beast}, we know that $-\Sigma_2(K)=\Sigma_2(-K)=\partial X$ where $X$ is a smooth, compact, oriented $4$-manifold $X$ with $b_2(X)=2m$ and positive definite intersection form of half-integer surgery type. Let $u_1, v_1, u_2, v_2, \dots, u_m, v_m$ denote generators for $H_2(X;\Z)$ satisfying $u_i \cdot u_j=2 \delta_{i,j}$ and  $u_i \cdot v_j= \delta_{i,j}$ for all $1 \leq i, j \leq m$. 

Consider the union $W\cup X$. This is a smooth, closed $4$-manifold with positive definite intersection form, and therefore by Donaldson's Theorem~A~\cite{donaldson} the intersection form is standard, i.e.\ given by the identity matrix. Moreover, since $W$ and $X$ are glued along a rational homology sphere, we have a lattice embedding
\[ 
\varphi\colon \big(H_2(W) \oplus H_2(X), Q_W\oplus Q_X) \hookrightarrow (\Z^N, \Id),
\]
where $N:= 2m+\sum_{i=1}^n 2a_i$ and we know that $(H_2(W), Q_W)\cong \bigoplus_{i=1}^n(H_2(W(a_i)),Q_{W(a_i)})$. 
Denote the standard basis for $\Z^N$ by 
\begin{equation}\label{eq:basis}
	e_1^1, \dots, e_1^{2a_1}, e_2^1, \dots, e_2^{2a_2}, \dots, e_n^1, \dots, e_n^{2a_n}, f_1, g_1, \dots, f_m, g_m.
\end{equation}

\begin{step}\label{step:alternating-1}
Basis elements of square $2$. 
\end{step} 

Fix $i, j$ and let $\varphi(x_i^j) = \sum_{k=1}^N c_k \omega_k$ where $\{\omega_1, \dots, \omega_N\}=:\mathcal{B}$ is a relabeling of the standard basis in \eqref{eq:basis}. Because this basis is orthonormal, we have 
\[2= \langle \varphi(x_i^j), \varphi(x_i^j) \rangle = \sum_{k=1}^N c_k^2.\]
On the other hand, there is only one way to write 2 as a sum of squares of nonzero integers:

\begin{equation}\label{eq:2=2}
	2=1+1.
\end{equation}
Therefore, $\varphi(x_i^j)=\pm \omega_{r_i^j} \pm \omega'_{r_i^j}$ for some $\omega_{r_i^j} \neq \omega'_{r_i^j}\in \mathcal{B}$. In other words, up to sign, each $\varphi(x_i^j)$ is a sum of two standard basis elements. Next, we consider the overlap among the standard basis elements that form each image $\varphi(x_i^j)$ as we vary $i,j$. We have
\begin{align*}
 \langle \varphi(x_i^j), \varphi(x_k^{l}) \rangle  &=
  \begin{cases}
   2        & \text{if } i=k \text{ and }j=l \\
   1        & \text{if } i=k \text{ and }|j-l|=1 \\
   0        & \text{otherwise} 
  \end{cases}
\end{align*}
Consider the case $\langle\varphi(x_i^j),\varphi(x_k^l)\rangle=0$. If  $\varphi(x_i^j)= \omega_{r_i^j} + \omega'_{r_i^j}$, since $\varphi(x_k^l)$ lies in the orthogonal complement of $\varphi(x_i^j)$, it follows that either $\varphi(x_k^l)= \pm(\omega_{r_i^j} - \omega'_{r_i^j})$, or $\varphi(x_k^l)= \omega_{r_k^l} + \omega'_{r_k^l}$ with $\{\pm\omega_{r_i^j}, \pm\omega'_{r_i^j} \}\cap\{\pm\omega_{r_k^l}, \pm\omega'_{r_k^l}\}=\emptyset$. In other words, the plane spanned by $\{\omega_{r_i^j}, \omega'_{r_i^j} \}$, the two standard basis vectors in the image $\varphi(x_i^j)$, and the plane spanned by $\{\omega_{r_k^l}, \omega'_{r_k^l}\}$, the two basis vectors in the image $\varphi(x_k^l)$, are either the same or intersect only at the origin. Upon a moment's thought, we can exclude the first case, since, whenever $x_i^j\neq x_k^l$, there exists an element $x_s^t$ whose image under $\varphi$ is also a sum of two standard basis elements and such that $x_s^t$ is adjacent to one of $x_i^j$ or $x_k^l$ but not the other (this uses that $a_i\geq 3$). 
In sum, for $x_i^j\neq x_k^l$, the images $\varphi(x_i^j)$ and $\varphi(x_k^l)$ share a basis element, up to sign, if and only if $i=k$ and $|j-l|=1$. Putting this together, we may assume, up to changing the names and signs of the elements of $\mathcal{B}$, that 
\[\varphi(x_i^j)= e_i^j -e_i^{j+1} \text{ for } 1 \leq i \leq n \text{ and } 1 \leq j \leq 2a_i-1.\]
We carry out a similar analysis regarding the images $\varphi(u_r)$. We use the fact that each $u_r$ squares to 2, lies in the orthogonal complement of $\{x_i^j\}_{i,j}\cup \{u_s\}_{s\neq r}$, and satisfies $\langle \varphi(u_i), \varphi(v_j)\rangle=\delta_{ij}$. This allows us to conclude, again up to sign and reordering the basis elements, that 
\begin{equation}\label{eq:short-stick-adjacent}
\varphi(u_i)= f_i - g_i \text{ for } i=1, \dots, m. 
\end{equation}
For more details on the conclusion \eqref{eq:short-stick-adjacent}, specifically why there is no overlap among the generators in the images $\{\varphi(u_r)\}$, see \cref{cor:specialucp2}.

For all $i,j$, since $\varphi$ is a lattice embedding, we have that $\langle \varphi(y_i),\varphi(x_j^{k})\rangle$ is $1$ precisely if $i=j$ and $k=a_i-1$, and $0$ in every other case. Similarly, $\langle \varphi(y_i),\varphi(u_j)\rangle=0$ for all $i,j$. As a result, we have some representation  
\begin{equation}\label{eq:w_i}
	\varphi(y_i)= 
\sum_{j\neq i} \left( b_i^j \sum_{k=1}^{2a_j} e_j^k \right) +
c_i \sum_{k=1}^{2a_i-1} e_i^k + (c_i+1)e_i^{2a_i}+ 
 \sum_{l=1}^m d_i^l (f_l+ g_l)
\end{equation}
for some choice of $b_i^j$ for $j \neq i$, $c_i$, and $d_i^l$ for $1 \leq l \leq m$, and all $i$. 

\begin{step}
Basis elements of square $3$. 
\end{step}

Since $\langle \varphi(y_i) , \varphi(y_i)\rangle= 3$, we may refer to the representation \eqref{eq:w_i} to see that for each $i=1, \dots, n$
\[
\sum_{j\neq i} 2(b_i^j)^2 a_j + c_i^2(2a_i-1)+  (c_i+1)^2 + 2 \sum_{l=1}^m (d_i^l)^2=3
\]
Note that all terms on the left-hand side of this equation are non-negative. 
If $c_i\neq 0$, then 
$c_i^2(2a_i-1)>3$ since $a_i \geq 3$. Therefore, $c_i=0$ and our equation becomes
\[
\sum_{j\neq i} (b_i^j)^2 a_j  +  \sum_{l=1}^m (d_i^l)^2=1\]
By assumption, $a_j \geq 3$ for all $j$ and so $b_i^j=0$ for all $j \neq i$. As a consequence we have that 
\[\sum_{l=1}^m (d_i^l)^2=1.\] Returning now to~\eqref{eq:w_i} and simplifying, we derive that 
\begin{equation}\label{eq:nice-w}
\varphi(y_i)= e_i^{2a_i}+ \sum_{l=1}^m d_i^l (f_l+ g_l).
\end{equation}
Since $\langle \varphi(y_i), \varphi(y_j)\rangle=0$ for $i\neq j$, the representation \eqref{eq:nice-w} implies that 
\[
\sum_{l=1}^m d_l^i d_l^j=0 \text{ for all } i \neq j.
\]
In other words, $(d_1^1, \dots, d_1^m), (d_2^1, \dots, d_2^m), \dots, (d_n^1, \dots, d_n^m)$ is a list of $n$ orthonormal vectors in $\mathbb{Z}^m$. So $m \geq n$ as desired. 
\end{proof}

\begin{remark}
As previously noted, the knot $K_1$ is the figure eight knot, and the knot $K_2$ is the slice stevedore knot. Therefore, we can complete the statement of \cref{prop:connectedsumsoftwist} by the facts that, for $n\geq 1$, 
\begin{align*}
\ucpbar(\#_{i=1}^n K_{1})&=\begin{cases}
0 & \text{if }n\text{ even}\\
1 & \text{if }n\text{ odd}
\end{cases}
\\[2mm]
\ucpbar(\#_{i=1}^n K_{2})&=0.
\end{align*}
For the first statement above we used that $K_1\# K_1$ is slice and $\ucpbar(K_1)=1$ since it can be changed to the unknot by a single negative to positive crossing change.
\end{remark}

We now illustrate the use of the proposition above to find new examples of knots with finite $\ucp$ and $\ucpbar$. 
Knots which are both positively slice and negatively slice were called \emph{biprojectively $H$-slice} in~\cite{manolescu-piccirillo}.

\begin{example}\label{cor:specialucp2}
Let $K= K_3 \#K_5$. We will show that $\ucp(K)$ and $\ucpbar(K)$ are both finite and strictly greater than $1$. By \cref{prop:connectedsumsoftwist}, we have that $\ucpbar(K)=2$. We also know that \[\ucp(K) \leq \ucp(K_3)+ \ucp(K_5)\leq 1+3=4.\] 

It remains to show that $\ucp(K)>1$. As in the proof of \cref{prop:connectedsumsoftwist}, we know that $\Sigma_2(K_a)$ is the lens space $L(4a+1, 2)$ and $\frac{4a+1}{2}= (2a+1)-\frac{1}{2}= [2a+1, 2]^-$. Then we see that $\Sigma_2(K)$ is the boundary of the plumbed negative definite $4$-manifold $W$ depicted in \cref{fig:k3plusk5}, where we denote generators of $H_2(W;\Z)$ by $\{x_1,x_2,w_1,w_2\}$, as shown in the figure, and orient them so that adjacency corresponds to intersection $+1$, i.e.\ $Q_W(x_1,x_2)=Q_W(w_1,w_2)=1$.  

\begin{figure}[htb]
\begin{tikzpicture}
\node[anchor=south west,inner sep=0] at (0,0){\includegraphics[width=5cm]{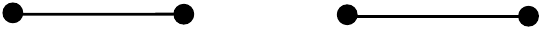}};
\node at (0,0.5) {\blue{$x_2$}};
\node at (1.6,0.5) {\blue{$x_1$}};
\node at (3.2,0.5) {\blue{$w_2$}};
\node at (4.8,0.5) {\blue{$w_1$}};
\node at (0,-0.2) {$-7$};
\node at (1.6,-0.2) {$-2$};
\node at (3.2,-0.2) {$-11$};
\node at (4.8,-0.2) {$-2$};
\end{tikzpicture}
\caption{A plumbing diagram for the 4-manifold $W$ with boundary $\partial W= L(13, 2)\# L(21, 2)= \Sigma_2(K_3 \# K_5)$. Adjacent generators have intersection $+1$.}\label{fig:k3plusk5}
\end{figure}
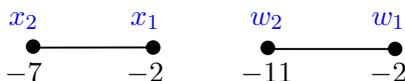

Suppose for the sake of a contradiction that $K_3 \#K_5$ is $H$-slice in $\CP^2$. Then by \cref{thm:beast} applied to $-K$, as in \cref{cor:beastmate}, there exists a smooth, compact, oriented $4$-manifold $X$ with boundary $-\Sigma_2(K)$ with $b_2(X)=2$ and negative definite intersection form of negative half-integer surgery type. Let $u,v$ denote the generators as in \cref{def:half-int}, with $u\cdot u=-2$ and $u\cdot v=1$. Consider the union $W\cup X$. This is a smooth, closed $4$-manifold with negative definite intersection form, and therefore by Donaldson's Theorem A~\cite{donaldson}, the intersection form is given by the negative of the identity matrix. So we have a lattice embedding 
\[
\varphi\colon \big(H_2(W) \oplus H_2(X), Q_W\oplus Q_X) \hookrightarrow (\Z^6, -\Id).
\]
Denote the standard generators of $\mathbb{Z}^6$ by $e_1, e_2, \dots, e_6$. Up to permuting the basis elements and replacing some $e_i$ with $-e_i$, we can assume that $\varphi(x_1)= e_1-e_2$, $\varphi(w_1)=e_3-e_4$, and $\varphi(u)= e_5-e_6$. To see this, assume for a contradiction that $\varphi(x_1)=e_1-e_2$ and $\varphi(w_1)=e_1+e_2$. Then we know that $\langle \varphi(x_1),\varphi(x_2)\rangle=1$ while $\langle \varphi(w_1),\varphi(x_2)\rangle=0$. Thus, subtracting the second equation from the first we have $1=-2\langle e_2, \varphi(x_2)\rangle$, which is a contradiction. The remaining cases can be similarly ruled out. 

We now know that for some integers $a,b,c,p,q,r$ we have  
\begin{align*}
\varphi(x_2)&= ae_1+(a+1)e_2+be_3+be_4+ce_5+ce_6, \\
\varphi(w_2)&= pe_1+pe_2 + q e_3+ (q+1)e_4+re_5+re_6.
\end{align*}
Since we must have $\langle \varphi(x_2), \varphi(x_2) \rangle= x_2 \cdot x_2=-7$, $\langle \varphi(w_2), \varphi(w_2) \rangle= w_2 \cdot w_2=-11$, and $\langle \varphi(x_2), \varphi(w_2) \rangle= x_2 \cdot w_2=0$, we obtain the following equations:
\begin{align*}
a^2+(a+1)^2+ 2b^2+2c^2&=7,\\
2p^2+q^2+(q+1)^2+2r^2&=11,\\
(2a+1)p+b(2q+1)+2cr&=0.
\end{align*}

For an integer $n$, the the values of $n^2+(n+1)^2$ are in the set $\{1, 5, 13, \dots\}$. By considering the first equation, we can conclude that $a^2+ (a+1)^2= 5$  and $b^2+c^2=1$. By considering the second equation, we can conclude that $q^2+(q+1)^2=1$ and $p^2+r^2=5$.
We therefore have $a \in \{-2, 1\}$, $\{b,c\}=\{0, \pm 1\}$ and $bc=0$, $q \in \{-1, 0\}$, and $\{p,r\}= \{\pm 1, \pm 2\}$. It follows that $(2a+1)p$ is a multiple of 3, exactly one of $b(2q+1)$ and $2cr$ is nonzero, and when non-zero, neither of $b(2q+1)$ and $2cr$ can be divisible by 3. So $(2a+1)p+b(2q+1)+2cr$ is not divisible by 3 and cannot be equal to zero.  This is our desired contradiction. 
\end{example}

Next we prove the second half of \cref{thm:finite-and-big}, i.e.\ we show that there are knots for which both $\ucp$ and $\ucpbar$ are arbitrarily large but finite. The following result gives two bounds on $\ucp(K)$ for a knot $K$, but as we will see in \cref{cor:bigucp2ucp2bar}, in practice we may apply one to $-K$ to obtain a bound on $\ucpbar(K)$. 

\begin{proposition}\label{prop:two-bounds}
Let $K= (\#_{i=1}^n K_{a_i}) \# (\#_{j=1}^m (-K_{b_j}))$, where $a_i, b_j \geq 3$ for all $i=1, \dots, n$ and $j=1, \dots, m$. Then $\ucp(K) \geq m-n$. 

If $b_j>a_i$ for all $i,j$, then $\ucp(K)\geq m$. 
\end{proposition}

\begin{proof}
As in the proof of \cref{prop:connectedsumsoftwist}, the double branched cover $\Sigma_2(K_{a_i})$ is the lens space $L(4a_i+1,2)$, which bounds the negative definite plumbed manifold $U(a_i)$ shown in \cref{fig:short-stick}, with the generators of $H_2(U(a_i);\Z)$ denoted by $\{z_i,w_i\}$ as given in the figure, and oriented so that $Q_{U(a_i)}(z_i,w_i)=1$. We also saw that $\Sigma_2(-K_{b_j})=-\Sigma_2(K_{b_j})$ is the lens space $L(4b_j+1, 4b_j-1)$, which bounds the negative definite plumbed $4$-manifold $V(b_j)$ shown in \cref{fig:long-stick}; for $j= 1\dots m$, the generators $\{x_j^1, \dots, x_j^{2b_j-1}, y_j\}$ have squares as given in the figure and are oriented so that adjacency corresponds to intersection $+1$. Let $W$ be the boundary connected sum of \[\{U(a_1),U(a_2),\dots,U(a_n),V(b_1),V(b_2),\dots, V(b_m)\}.\]Observe that $b_2(W)=2n+ \sum_{j=1}^m 2b_j$. 

\begin{figure}[htb]
\begin{subfigure}[b]{0.49\textwidth}
\centering
\begin{tikzpicture}
\node[anchor=south west,inner sep=0] at (0,0){\includegraphics[width=2.3cm]{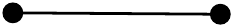}};
\node at (0.2,0.6) {\blue{$z_i$}};
\node at (2.2,0.6) {\blue{$w_i$}};
\node at (0.2,-0.25) {$-2a_i-1$};
\node at (2.2,-0.25) {$-2$};
\end{tikzpicture}
\caption{}\label{fig:short-stick}
\end{subfigure}\hfill 
\begin{subfigure}[b]{0.49\textwidth}
\centering
\begin{tikzpicture}
\node[anchor=south west,inner sep=0] at (0,0){\includegraphics[width=5cm]{Figure8alternative.pdf}};
\node at (-0,0.6) {\blue{$x_j^1$}};
\node at (1.5,0.6) {\blue{$x_j^2$}};
\node at (3.2,0.6) {\blue{$x_j^{2b_j-1}$}};
\node at (4.9,0.6) {\blue{$y_j$}};
\node at (0,-0.25) {$-2$};
\node at (1.7,-0.25) {$-2$};
\node at (3.2,-0.25) {$-2$};
\node at (4.9,-0.25) {$-3$};
\end{tikzpicture}
\caption{}\label{fig:long-stick}
\end{subfigure}
\caption{(a) A plumbing diagram for the manifold $U(a_i)$. (b) A plumbing diagram for the manifold $V(b_j)$. Adjacent generators in either diagram have intersection $+1$. }
\end{figure}

Now suppose that $K$ is $H$-slice in $\#^p \CP^2$. Then $-K$ is $H$-slice in $\#^p \ol{\CP^2}$ and so by \cref{cor:beastmate}, we know that $\Sigma_2(-K)=-\Sigma_2(K)=\partial X$ where $X$ is a smooth, compact, oriented $4$-manifold $X$ with $b_2(X)=2p$ and negative definite intersection form of negative half-integer surgery type. Let $u_1, v_1, u_2, v_2, \dots, u_p, v_p$ denote generators for $H_2(X;\Z)$ satisfying $u_i \cdot u_j=-2 \delta_{i,j}$ and  $u_i \cdot v_j= \delta_{i,j}$ for all $1 \leq i, j \leq p$. 

As before, we consider the union $W\cup X$. This is a smooth, closed $4$-manifold with negative definite intersection form, and therefore by Donaldson's Theorem~A~\cite{donaldson} the intersection form is standard, i.e.\ given by the negative of the identity matrix. Moreover, since $W$ and $X$ are glued along a rational homology sphere, we have a lattice embedding
\[ 
\varphi\colon \big(H_2(W) \oplus H_2(X), Q_W\oplus Q_X) \hookrightarrow (\Z^N, -\Id),
\]
where $N:= 2p+2n+ \sum_{j=1}^m 2b_j$ and we know that \[(H_2(W), Q_W)\cong \bigoplus_{i=1}^n(H_2(U(a_i)),Q_{U(a_i)}) \oplus \bigoplus_{i=1}^m(H_2(V(b_j)),Q_{V(b_j)}).\] 
We somewhat suggestively denote a standard basis for $\Z^N$ by 
\[ \bigcup_{i=1}^n \{e_i, f_i\} \cup \bigcup_{j=1}^m \{g_j^1, \dots, g_j^{2b_j}\} \cup \bigcup_{k=1}^p \{h_k, l_k\}.\]

\setcounter{step}{0}
\begin{step}
Basis elements of square $-2$.
\end{step}

We proceed as in Step~\ref{step:alternating-1} of the proof of~\cref{prop:connectedsumsoftwist}, using \eqref{eq:2=2} with each side multiplied by $-1$. Since each $w_i$, $x_j^\ell$, and $u_k$ has square $-2$, and by considering their relative intersections, we may assume without loss of generality (see e.g.\ \cref{cor:specialucp2}) that 
\begin{alignat*}
\varphi(w_i)&=e_i- f_i &&\text{ for } i=1, \dots, n, \\
\varphi(x_j^{\ell})&= g_j^{\ell}-g_j^{\ell+1} &&\text{ for } j=1, \dots, m \text{  and }\ell=1, \dots, 2b_j-1, \\
\varphi(u_k)&=h_k-l_k &&\text{ for } k=1, \dots p.
\end{alignat*}

For each $i=1, \dots, n$, by considering the intersection of $z_i$ with the $x_*$, $v_*^*$, and $y_*$ we see that there exist integers $\alpha_i^c$ for $c=1, \dots, n$, $\beta_i^j$ for $j=1, \dots, m$, and $\gamma_i^k$ for $k=1, \dots, p$ such that
\begin{align}\label{eqn:phiwiorig}
 \varphi(z_i)=
\sum_{q=1}^n \alpha_i^q (e_q+f_q) + f_i  + \sum_{r=1}^m\left( \beta_i^r \sum_{\ell=1}^{2b_r} g_r^ \ell \right) + \sum_{k=1}^p \gamma_i^k (h_k+l_k).\end{align}
Similarly, for each $j=1, \dots, m$ by considering the intersection of $y_j$ with the $x_*$, $v_*^*$, and $y_*$, we can see that there exist integers $\delta_j^s$ for $s=1, \dots, n$, $\epsilon_j^d$ for $d=1, \dots, m$, and $\theta_j^k$ for $k=1, \dots, p$ such that
\begin{align}\label{eqn:phiujorig}
\varphi(u_j)= \sum_{s=1}^n \delta_j^s(e_s+f_s)+ \sum_{t=1}^m\left( \epsilon_j^t \sum_{\ell=1}^{2b_t} g_t^{\ell} \right) + g_j^{2b_j} + \sum_{k=1}^p \theta_j^k (h_k + l_k).
\end{align}

\begin{step}
Basis elements of square $-3$
\end{step}

Observe that for each $j=1, \dots, m$ we have that 
\begin{align}\label{eqn:ujdotuj}
3=- \langle\varphi(y_j),\varphi(y_j)\rangle =\sum_{s=1}^n 2(\delta_j^s)^2 + \sum_{t=1}^m 2b_t (\epsilon_j^t)^2 + (2\epsilon_j^j+1) + \sum_{k=1}^p 2(\theta_j^k)^2.
\end{align}

Since each $b_t \geq 3$, we must have that $\epsilon_j^t=0$ for all $t=1, \dots m$. (This takes a little more thought for $t=j$, where the contribution is $2b_j (\epsilon_j^j)^2+ 2 \epsilon_j^j+1= (2b_j-1) (\epsilon_j^j)^2+ (\epsilon_j^j+1)^2$.) 

The equation \eqref{eqn:ujdotuj} therefore simplifies to yield
\begin{align*}
1= \sum_{s=1}^n (\delta_j^s)^2 + \sum_{k=1}^p (\theta_j^k)^2 \text{ for all } j=1, \dots, n. 
\end{align*}
We conclude that for each $j$ the set $\{\delta_j^s\}_{s=1}^n \cup \{\theta_j^k\}_{k=1}^p$ has exactly one nonzero entry, which equals~$\pm 1$.  It follows that fixing $j$ and letting $j', s, k$ vary, at most one of the products  $\{ \delta_j^s \delta_{j'}^s \, , \, \theta_j^k \theta_{j'}^k\}$ can be nonzero.
We can also observe that for all $j \neq j'$ we have
\begin{align*}
0=\frac{1}{2}\langle \varphi(y_j), \varphi(y_{j'})\rangle= \sum_{s=1}^n \delta_j^s \delta_{j'}^s + \sum_{k=1}^p \theta_j^k \theta_{j'}^k.
\end{align*}
 Therefore, for all $j' \neq j$, if $\delta_j^s \neq 0$, then $\delta_{j'}^s=0$ and similarly, if $\theta_j^k \neq 0$, then $\theta_{j'}^k=0$. The pigeonhole principle therefore allows us to conclude that $m \leq n+p$; that is, we have established our first claim by showing $p \leq m-n$.\\

For the proof of the second half of the statement of the proposition, assume that  $b_j>a_i$ for all $i,j$. We will now show that $m \leq p$. By another application of the pigeonhole principle note that we will be done if we can prove that $\delta_j^s=0$ for all $j=1, \dots m$ and $s=1, \dots, n$.  
So suppose for a contradiction that for some $1 \leq j \leq m$ and $1 \leq i \leq n$ we have that $\delta_{j}^{i} \neq 0$. Returning to \eqref{eqn:phiujorig} and recalling our observation that $\epsilon_j^t=0$ for all $t$, we therefore have that 
\begin{align}\label{eqn:phiujfinal}
\varphi(u_j)= \pm(e_i+f_i) +g_j^{2b_j}.\end{align}

\begin{step}
Considering $\varphi(z_i)$.
\end{step}

By \eqref{eqn:phiwiorig} we have that
\[2a_i+1= -\langle \varphi(z_i), \varphi(z_i)\rangle
=\sum_{q=1}^n 2(\alpha_i^q)^2+ (2 \alpha_i^i+1) + \sum_{r=1}^m 2b_r(\beta_i^r)^2 + \sum_{k=1}^p 2(\gamma_i^k)^2.
\]
Since $b_r>a_i$ for all $r=1, \dots, n$ we have that $2b_r>2a_i+1$ and hence that each $\beta_i^r=0$, and further that 
\begin{align}\label{eqn:phiwi}
\varphi(z_i)= \sum_{r=1}^n \alpha_i^r (e_r+f_r) + f_i + \sum_{k=1}^p \gamma_i^k (h_k+l_k).
\end{align}
By \eqref{eqn:phiujfinal} and~\eqref{eqn:phiwi}, we therefore have that
\begin{align*}
0=\langle \varphi(w_i), \varphi(u_j)\rangle
&=\left\langle\left(\sum_{r=1}^n \alpha_{i}^r (e_r+f_r) + f_j + \sum_{k=1}^p \gamma_{i}^k (h_k+l_k)\right), \left(\pm(e_i+f_i)+ g_j^{2b_j}\right)\right\rangle\\
&= \left\langle(\alpha_{i}^{i} e_{i} + (\alpha_{i}^i +1) f_i), \pm(e_i+f_i)\right\rangle\\
&= \mp (2 \alpha_i^i + 1).
\end{align*}
Since $\alpha_i^i \in \Z$,  we obtain our desired contradiction. 
\end{proof}

As promised, the above result allows us to bound $\ucp$ and $\ucpbar$ simultaneously. 

\begin{corollary}\label{cor:bigucp2ucp2bar}
Let $K= (\#_{i=1}^n K_{a_i}) \# (\#_{j=1}^m (-K_{b_j}))$, where $b_j>a_i\geq 3$ 
for all $i$ and $j$.
Then $\ucp(K) \geq m$ and $\ucpbar(K) \geq n-m$.
\end{corollary}

\begin{proof}
By the second part of \cref{prop:two-bounds}, we know that $\ucp(K)\geq m$. Consider $-K:=(\#_{j=1}^m K_{b_j}) \# (\#_{i=1}^n (-K_{a_i})) $. Applying the first part of \cref{prop:two-bounds} to $-K$, we see that 
\[\ucp(-K)\geq n-m,\] implying that $\ucpbar(K)\geq n-m$, as needed.
\end{proof}

The bounds above are optimal in the sense that there is no better bound for this family that only depends on $n$ and $m$. For example, in the case that $m=0$ or $n=0$, we recover \cref{prop:connectedsumsoftwist}. However, these are far from the natural upper bounds, as we show in the next example. 

\begin{example}\label{BPH-ex}
Let $\displaystyle K= (p+q) K_3-p K_4$ for $p,q \geq 0$. Then $p \leq \ucp(K)\leq 2p+q$ and $q \leq \ucpbar(K) \leq p+q$, as follows. The lower bounds come from \cref{cor:bigucp2ucp2bar}. Note that $K_3$ can be transformed into the stevedore knot $K_2$, which is slice, by a single positive to negative crossing change, and that $-K_a$ can be unknotted by a single positive to negative crossing change for any $a \in \Z$. 
Therefore 
\begin{equation*}
p \leq \ucp(K)  \leq (p+q) \ucp(K_3)+  p\ucp(-K_4)=2p+q.
\end{equation*}
We also have that $K_3$ can be transformed into $K_4$ (and hence $K_3\#-K_4$ into the slice knot $K_4 \# -K_4$) via a single negative to positive crossing change and that $K_a$ can be unknotted by a single negative to positive crossing change, and so
\begin{equation*}
q \leq \ucpbar(K) \leq p\ucpbar(K_3-K_4)+q \ucpbar(K_3)=p+q. 
\end{equation*}
\end{example}
\smallskip

We finish this section by giving the proof of \cref{thm:finite-and-big}, which we first restate. 

\begin{reptheorem}{thm:finite-and-big}
For any $n\geq 0$, there exists a knot $K$ such that $n\leq \ucp(K)< \infty$ and $n\leq \ucpbar(K)< \infty$. 
Also,  for any $n\geq 0$ there exists a knot $J$ with trivial signature function and $\ucp(J)=n$.
\end{reptheorem}

\begin{proof}The first  sentence is proved by \cref{cor:bigucp2ucp2bar}.
The second sentence is proved by~\cref{prop:connectedsumsoftwist}, by taking the mirror images of the knots, i.e.\ we take $K:= \#^n_{i=1} \ol{K_{a_i}}$ where $a_i\geq 3$ for all $i$. That the signature function is trivial can be seen from direct computation, as follows. Briefly, the knot $K_{a_i}$ has Seifert matrix $M_{a_i}:=\left[\begin{array}{cc} 1	&1\\ 0	&-a_i\end{array}\right]$, and it is easy to compute that $M_{a_i}+M_{a_i}^T$ has trivial signature for $a_i\geq 1$. Since $K_{a_i}$ has genus one, it follows that the signature function is also trivial. Moreover, as previously mentioned, an infinite family of the knots $\{K_{a_i}\}$ for $a_i\geq 3$ are algebraically slice, as shown e.g.\ in~\cite{CG86}*{p.\~182}. Since the signature function is additive under connected sum, the claim follows. 
\end{proof}

\section{Pretzel knots}\label{sec:pretzels}
In this section we will apply \cref{thm:beast} to derive some bounds on $\ucp$ for certain families of pretzel knots. We begin by recalling a result of Bryant, which will enable us to easily locate pretzel knots with trivial ordinary signature. 

\begin{proposition}[\cite{KB}*{Theorem~2.1}]\label{prop:Bryant}
Let $K=P(q_1, \dots, q_{p+n})$, for $p,n\geq 0$, be a pretzel knot for which $p$ of the parameters $q_i$ are positive, $n$ are negative, and all are odd. 
Then the signature of $K$ is given by 
\[\sigma(K)=p-n-\sgn\left(\frac{1}{q_1}+\dots+\frac{1}{q_{p+n}}\right),\] where $\sgn$ denotes the sign function $\sgn(x):=|x|/x$ for $x \neq 0$ and $\sgn(0):=0$. 
\end{proposition}

Next we give our most general restriction on positively slice pretzel knots.

\begin{proposition}\label{prop:equations} 
Fix $k\geq 1$. Let $K=P(p_1, \dots, p_{k+1}, q_1, \dots q_{k})$ be a pretzel knot such that each $p_i$ and each $q_j$ is odd. Suppose further that $p_i\geq 3$ and $q_j \leq -3$
for all $i,j$, and $\sigma(K)=0$. 

If $K$ is $H$-slice in $\#^m \CP^2$ then there exist integers $a_i^j$ and $b_\ell^j$, for $1 \leq i \leq k+1$, $1 \leq j \leq {k}$, and $1 \leq \ell \leq m$, such that the following holds.
\begin{enumerate}
\item\label{item:pretzel-1} For all $1 \leq j \leq {k}$,
\[\sum_{i=1}^{k+1} a_i^j=1 \text{ and } \sum_{i=1}^{k+1} (a_i^j)^2 p_i +2  \sum_{\ell=1}^m (b_\ell^j)^2 = |q_j|.\]
\item\label{item:pretzel-2} For all $1 \leq j' \neq j \leq {k}$, 
\[ \sum_{i=1}^{k+1} a_i^j a_i^{j'} p_i + 2 \sum_{\ell=1}^m b_\ell^j b_\ell^{j'}=0.\]
\end{enumerate}
\end{proposition}

\begin{proof}
Since $\sigma(K)=0$, we know from~\cref{prop:Bryant} that 
\[
\frac{1}{p_1}+\dots+\frac{1}{p_{k+1}}+ \frac{1}{q_1}+\dots+\frac{1}{q_{k}}>0.\]
Then as in \cref{ex:pretzel-cover}, by the assumptions on the parameters $\{p_i\}$ and $\{q_j\}$, there is a compact, connected, oriented, plumbed $4$-manifold $W$, with negative definite intersection form and plumbing graph given in \cref{fig:pretzel-plumbing}, such that $\partial W =\Sigma_2(K)$. We label the vertices of the plumbing diagram for $W$ as in \cref{fig:pretzel-plumbing}, and orient them so that adjacency corresponds to intersection $+1$. Specifically, we have generators \[x_1^1, \dots, x_1^{p_1-1}, x_2^1, \dots, x_2^{p_2-1}, \dots, x_{k+1}^1, \dots, x_{k+1}^{p_{k+1}-1}, y, z_1, z_2, \dots, z_{k}, \] where each $x_i^j$ is $-2$-framed, each $z_i$ is $q_i$-framed, and the vertex $y$ is $-(k+1)$-framed. Note that the rank of $H_2(W)$ is $\sum_{i=1}^{k+1} (p_i-1)+k+1=\sum_{i=1}^{k+1} p_i$.
\begin{figure}[htb]
\centering
\begin{tikzpicture}
\node[anchor=south west,inner sep=0] at (0,0){\includegraphics[width=11.5cm]{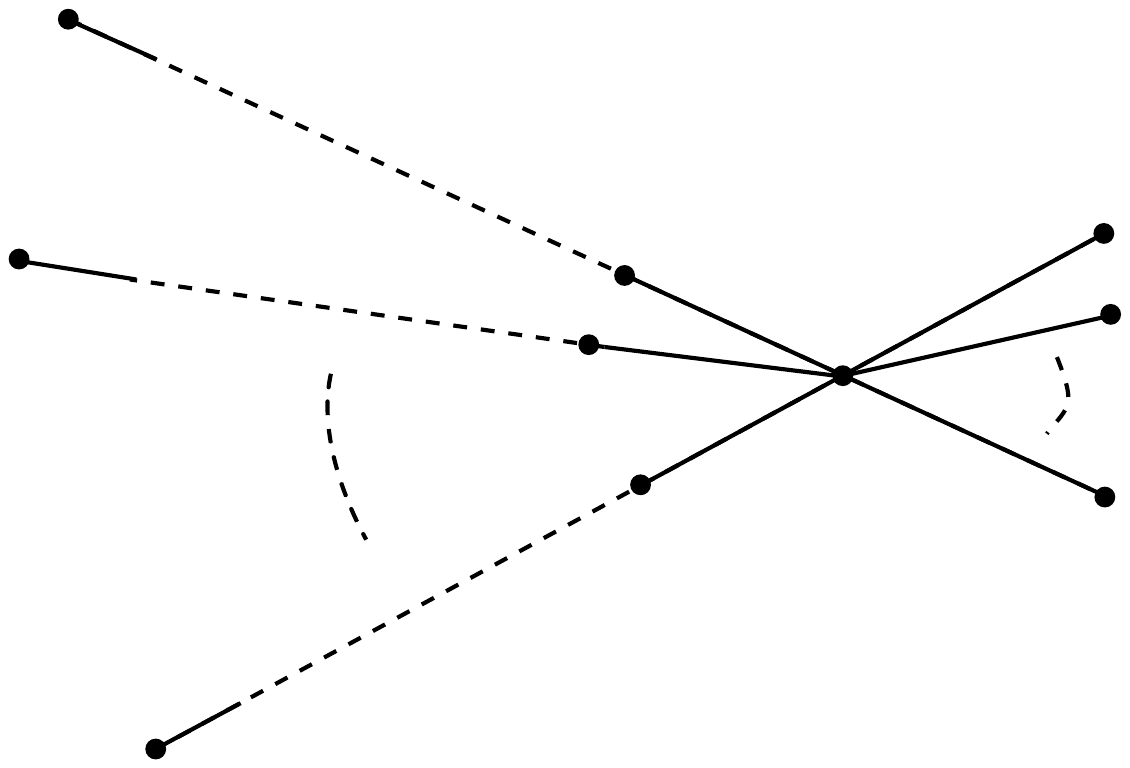}};
\node at (0.5,8.1) {\blue{$x^1_1$}};
\node at (0.4,7.2) {$-2$};
\node at (0.1,5.7) {\blue{$x^1_2$}};
\node at (0,4.75) {$-2$};
\node at (1.25,0.55) {\blue{$x^1_{k+1}$}};
\node at (1.5,-0.3) {$-2$};
\node at (6.55,5.5) {\blue{$x^{p_1-1}_1$}};
\node at (6.35,4.65) {$-2$};
\node at (5.6,4.75) {\blue{$x^{p_2-1}_2$}};
\node at (6,3.9) {$-2$};
\node at (6,3.1) {\blue{$x^{p_{k+1}-1}_{k+1}$}};
\node at (6.5,2.5) {$-2$};
\node at (8.7,4.35) {\blue{$y$}};
\node at (8.5,3.35) {$-(k+1)$};
\node at (11,5.75) {\blue{$z_1$}};
\node at (11.25,4.9) {\blue{$z_2$}};
\node at (11.5,3.1) {\blue{$z_k$}};
\node at (11.65,5.3) {$q_1$};
\node at (11.5,4.3) {$q_2$};
\node at (11.25,2.4) {$q_k$};
\end{tikzpicture}
\caption{A plumbing diagram for a $4$-manifold $W$ with $\partial W= \Sigma_2(K)$, where $K=P(p_1, \dots, p_{k+1}, q_1, \dots, q_k)$, with each $p_i$ positive and each $q_j$ negative. Adjacent generators have intersection $+1$.}
\label{fig:pretzel-plumbing}
\end{figure}

Since $K$ is $H$-slice in $\#^m \CP^2$, we know that $-K$ is $H$-slice in $\#^m \ol{\CP^2}$. So by \cref{cor:beastmate}, we know that $\Sigma_2(-K)=-\Sigma_2(K)=\partial X$ where $X$ is a smooth, compact, oriented $4$-manifold $X$ with $b_2(X)=2m$ and negative definite intersection form of negative half-integer surgery type. Let $u_1, v_1, u_2, v_2, \dots, u_m, v_m$ denote generators for $H_2(X;\Z)$ satisfying $u_i \cdot u_j=-2 \delta_{i,j}$ and  $u_i \cdot v_j= \delta_{i,j}$ for all $1 \leq i, j \leq m$. 

As before, we consider the union $W\cup X$. This is a smooth, closed $4$-manifold with negative definite intersection form, and therefore by Donaldson's Theorem~A~\cite{donaldson} the intersection form is standard, i.e.\ given by the negative of the identity matrix. Moreover, since $W$ and $X$ are glued along a rational homology sphere, we have a lattice embedding
\[ 
\varphi\colon \big(H_2(W) \oplus H_2(X), Q_W\oplus Q_X) \hookrightarrow (\Z^N, -\Id),
\]
where $N:= 2m+\sum_{i=1}^{k+1} p_i$. Denote the standard basis for $\mathbb{Z}^{N}$ by 
\begin{equation*}
	e_1^1, \dots, e_1^{p_1}, e_2^1, \dots, e_2^{p_2}, \dots, e_{k+1}^1, \dots, e_{k+1}^{p_{k+1}}, f_1, g_1,\dots, f_{m}, g_m.
\end{equation*}

\setcounter{step}{0}
\begin{step} Basis elements of square $-2$. 
 \end{step}
As in~\cref{step:alternating-1} of the proof of \cref{prop:connectedsumsoftwist}, we assume without loss of generality that 
\begin{alignat*}
\varphi(x_i^j)&=e_i^j- e_i^{j+1} &&\text{ for } i=1, \dots, k+1 \text{ and }1 \leq j \leq p_i-1,\\
\varphi(u_\ell)&= f_{\ell}-g_\ell &&\text{ for }  \ell=1, \dots m.
\end{alignat*}
Note that we are using the assumption that $p_i\geq 3$ to conclude that the basis elements in the images $\varphi(x_i^j)$ and $\varphi(x_k^l)$ are disjoint when $i\neq k$. 

\begin{step}\label{step:pretzel-2}
Considering $\varphi(y)$.
\end{step}

We have that $\varphi(x_i^{p_i-1})=e_i^{p_i-1} - e_i^{p_i}$ and $\langle  \varphi(y), \varphi (x_i^{p_i-1})  \rangle=1$, for all $i$. Thus,
\begin{equation} \label{eq:off-by-1}
	\langle  \varphi(y), \varphi (x_i^{p_i-1})  \rangle =\langle  \varphi(y), e_i^{p_i-1} \rangle - \langle  \varphi(y), e_i^{p_i}\rangle = 1,
\end{equation}
 that is, the coefficients of $e_i^{p_i-1}$ and $e_i^{p_i}$ in  $\varphi(y)$ differ by 1. In particular, for any $i=1, 2, \dots, k+1$, they are not both 0. 
 At the same time, for $1\leq j \leq p_i-2$, $$\langle  \varphi(y), \varphi(x_i^{j}) \rangle = \langle  \varphi(y), e_i^j \rangle - \langle  \varphi(y),  e_i^{j+1} \rangle = 0.$$ 
Therefore, we can inductively conclude that, for $j, j'\in \{1, \dots, p_i-2\}$, 
\begin{equation}\label{eq:equal-coeffs}
	\langle  \varphi(y), e_i^j \rangle = \langle  \varphi(y), e_i^{j'} \rangle.
\end{equation}
 In other words, for $1\leq j \leq p_i-2$, the coefficient of $e_i^j$ in $ \varphi(y)$ is determined by $i$ and is independent of $j$.
This reasoning applies to all $i\in\{1, \dots, k+1\}$. We also know that  
\begin{equation} \label{eq:y-sqared}
	\langle \varphi(y), \varphi(y) \rangle = -k-1.
\end{equation}
This equation would be violated if for any $i\in\{1, \dots, k+1\}$ and any $j\in\{1,\dots, p_i-1\}$ we had $\langle  \varphi(y), e_i^j \rangle \neq 0$, due to \eqref{eq:off-by-1} and~\eqref{eq:equal-coeffs}. We thus conclude for each $i$ that $\langle  \varphi(y), e_i^j \rangle = 0$ for $j\in\{1,\dots, p_i-1\}$ and $\langle  \varphi(y), e_i^{p_i} \rangle = -1$. To summarize, we have that 
\begin{equation}\label{eqn:y}
\varphi(y)=-\sum_i^{k+1}e_i^{p_i}.
\end{equation}
\begin{step} Considering $\varphi(z_j)$. 
\end{step}
For $j=1, \dots, k$,  let
\begin{equation}\label{eqn:zj}
\varphi(z_j) = \sum_{i=1}^{k+1} \left(a_i^j \sum_{n=1}^{p_i} e_i^n \right) + \sum_{\ell=1}^{m} b_\ell^j(f_\ell+g_\ell)
\end{equation} 
for some parameters $a_i^j,  b_\ell^j$. As before, the coefficient of $e_i^n$ depends only on $i$ and not on $n$, since  $\langle \varphi(x_i), \varphi(z_j)\rangle = 0$ for all $i, j$. Similarly, the coefficients of $f_\ell$ and $g_\ell$ are equal since $\langle \varphi(w_i), \varphi(z_j)\rangle = 0$ for all $i, j$.

Now statement~\eqref{item:pretzel-1} follows from recalling that $\langle \varphi(z_j),\varphi(y)\rangle=1$ and 
$\langle\varphi(z_j), \varphi(z_j)\rangle= q_j$, for each $j$, and applying \eqref{eqn:y} and \eqref{eqn:zj}. 
Similarly, statement~\eqref{item:pretzel-2} follows from the fact that $\langle \varphi(z_j),\varphi(z_{j'})\rangle=0$. 
\end{proof}

Recall that the oriented diffeomorphism type of $\Sigma_2(K)$ for a pretzel knot $K$ does not depend on the order of the twisting parameters. More generally, mutant knots have (oriented) diffeomorphic double branched covers~\cite{viro}, so we have the following corollary.

\begin{corollary}\label{cor:mutant-equations}
Let $K$ be a mutant of the pretzel knot $P(p_1, \dots, p_{k+1}, q_1, \dots q_{k})$ where each $p_i$ and each $q_j$ is odd, and further assume that $p_i\geq 3$ and $q_j \leq -3$ for all $i,j$, and $\sigma(K)=0$. 

If $K$ is $H$-slice in $\#^m \CP^2$ then there exist integers $a_i^j$ and $b_\ell^j$, for $1 \leq i \leq k+1$, $1 \leq j \leq {k}$, and $1 \leq \ell \leq m$, such that the following statements hold.
\begin{enumerate}
\item\label{item:mutant-pretzel-1} For all $1 \leq j \leq {k}$,
\[\sum_{i=1}^{k+1} a_i^j=1 \text{ and } \sum_{i=1}^{k+1} (a_i^j)^2 p_i +2  \sum_{\ell=1}^m (b_\ell^j)^2 = |q_j|.\]
\item\label{item:mutant-pretzel-2} For all $1 \leq j' \neq j \leq {k}$, 
\[ \sum_{i=1}^{k+1} a_i^j a_i^{j'} p_i + 2 \sum_{\ell=1}^m b_\ell^j b_\ell^{j'}=0.\]
\end{enumerate}
\end{corollary} 

We can now apply the above to characterize when $3$-strand pretzel knots with odd parameters are positively slice.
Compare the equivalence of the conditions (1) and (2) below with~\cref{lem:slicetotwist}.

\begin{corollary}\label{cor:3strandposslice}
Let $K$ be a $3$-strand pretzel knot, none of whose parameters are $\pm 1$, and all of whose parameters are odd.
The following are equivalent.
\begin{enumerate}
\item $K$ is positively slice.
\item $K$ can be transformed to a slice knot by positive to negative crossing changes.
\item Either (a) $K=P(-q-2k,q,r)$ for odd $q,r\geq 3$  and $k \geq 0$ or (b) $K$ has at least two negative parameters. 
\end{enumerate}
\end{corollary}

\begin{proof}
We will eventually use~\cref{cor:mutant-equations}. However, we remind the reader that the (unoriented) isotopy class of $P(p,q,r)$ depends only on the set $\{p,q,r\}$, since cyclic permutations and the reversal permutation preserve the (unoriented) isotopy class of any pretzel knot and  
these permutations generate the symmetric group $S_3$. 
So in this case we could have used~\cref{prop:equations} directly.

First, note that (2) implies (1) by the discussion in~\cref{sec:background} (see e.g.~\cref{lem:slicetotwist}).  If (3)(a) holds, then $K$ can be transformed into the ribbon knot $P(-q,q,r)$ via $k$ positive to negative crossing changes.  If (3)(b) holds, then without loss of generality assume that $K=P(p,q,r)$ for $p,q<0$ and note that by increasing the parameter $p$ to $|q|$ via positive to negative crossing changes, we obtain the ribbon knot $P(|q|, q, r)$. Therefore, (3) implies (2).

Finally we show that (1) implies (3), by proving the contrapositive.  Note that if $K$ has three positive parameters then $\sigma(K)>0$ by \cref{prop:Bryant} and so $K$ is not positively slice. Assume $K=P(-p,q ,r)$ where $p, q,r>0$ and without loss of generality $q \leq r$.  Since we are not in Case (3)(a), 
we may assume that $p<q$, in order to show that $K$ is not positively slice. By \cref{prop:Bryant} we know that 
\[ \sigma(K)= 1 - \sgn(\frac{-1}{p}+ \frac{1}{q}+ \frac{1}{r}),\]
so either $\sigma(K)=2$ (in which case we would be done) or $\sigma(K)=0$. If $\sigma(K)=0$ and $K$ is $H$-slice in $\#^m\CP^2$, \cref{prop:equations} (or \cref{cor:mutant-equations}) implies there exist integers $a_q$, $a_r$, and $b_l$ such that 
\[ a_q+ a_r=1 \text{ and } a_q^2 q+ a_r^2 r+2  \sum_{l=1}^m b_l^2 = p.\] 
Since $p<q,r$, we must have $a_q=a_r=0$ from the second equation, which contradicts the first equation.
 We conclude that $K$ is not positively slice. 
\end{proof} 

As an immediate consequence, we obtain the following characterization of biprojectively slice odd 3-strand pretzel knots. 

\begin{corollary}~\label{cor:biprojectivelyslice}
Let $K$ be a $3$-strand pretzel knot, none of whose parameters are $\pm 1$, and all of whose parameters are odd.
Then the following statements are equivalent:
\begin{enumerate}
\item $K=P(p,q,r)$ is biprojectively slice 
\item $K$ can be transformed to a slice knot by either positive to negative or negative to positive crossing changes.
\item The set of parameters $\{p,q,r\}= \pm\{a, b, -a-c\}$ for some $a,b,c>0$. 
\end{enumerate}
\end{corollary}

Proposition~\ref{prop:equations} also allows us to make the following precise computations of $\ucp$. 

\begin{corollary}\label{cor:precise3strand}
Let $p\geq3$ and $r$ be odd integers.
\begin{enumerate}
\item\label{item:precise-1} For $r > p+2$, we have $\ucp(P(p, -p-2, r))=1$.
\item\label{item:precise-2} For $r>p+4$, we have $\ucp(P(p, -p-4, r))=2$. 
\item\label{item:precise-3} For $r>p+6$, we have $\ucp(P(p, -p-6, r))=3$. 
\end{enumerate}
\end{corollary}
\begin{proof}
For any $q=1,2,3$ and $r> p+2q$, define $K(q,r)$ to be the pretzel knot $P(p, -p-2q, r)$. Observe that $q$ positive to negative crossing changes transform the standard diagram of $K(q,r)$ into a diagram for $P(p, -p, r)$, which is ribbon. So $\ucp(K(q, r))\leq q$. 
Now, observe that $K$ has two positive parameters and one negative parameter, and also that 
\[ \frac{1}{p}+\frac{1}{-p-2q} + \frac{1}{r}> \frac{1}{r}>0,\]
so by \cref{prop:Bryant}, we have $\sigma(K)=0$. We now split into cases according to the value of $q$.

\begin{case} 
$q=1$
\end{case}

In this case, we only need to show that $K$ is not slice. This is originally due to Greene-Jabuka~\cite{greene-jabuka}, but note that we can reprove their result for these knots by appealing to \cref{cor:mutant-equations}: were $K$ to be slice in $B^4$, then there would exist an integer solution $\{a_1, a_2\}$ to 
\[a_1+a_2=1 \text{ and } a_1^2p+a_2^2r=p+2,\]
which since $r>p+2$ there is not. 
\begin{case}
$q=2$.
\end{case}
We assume for a contradiction that $K(2,r)$ is $H$-slice in $\CP^2$. By \cref{cor:mutant-equations}, this implies that there is an integer solution $\{a_1, a_2, b\}$ to 
\begin{align}
a_1+a_2=1  \text{ and } 
a_1^2p +a_2^2r + 2 b^2= p+4
\end{align}
By considering the second equation and recalling that 
$r>p+4$, 
we see that $a_2=0$. The first equation then implies that $a_1=1$, and the second equation then reduces to $2b^2=4$, which is our desired contradiction since $2$ is not the square of any integer.

\begin{case}
$q=3$
\end{case}

As before, assume for a contradiction that $K(3,r)$ is $H$-slice in $\#^2 \CP^2$, and conclude by \cref{cor:mutant-equations} that there must be an integer solution $\{a_1, a_2, b_1, b_2$\} to
\begin{align}
a_1+a_2=1 \text{ and } a_1^2p+a_2^2 r+ 2b_1^2+2b_2^2= p+6
\end{align}
As before, we  determine that $a_2=0$, so $a_1=1$. We now obtain that $3$ can be written as the sum of two squares, again a contradiction. 
\end{proof}
Since the knots of \cref{cor:precise3strand} have signature equal to zero and genus equal to one, they have vanishing Tristram-Levine signature function. Therefore, the bounds on $\ucp$ coming from~\cite{CN} are not able to prove this result. We also speculate that $\ucp(P(p, -p-2k, r))=k$, whenever $k\geq 0$ and for odd integers $p$ and $r$, with $p\geq 3$ and $r>p+2k$. See~\cref{prob:genus-one}. 

As mentioned in the introduction, we find knots $K$ for which we know exactly for which values of $m, n$ the knot $K$ is slice in $m \CP^2 \# n \overline{\CP}^2$ (cf.\ \cref{prob:indefinite}). 

\begin{corollary} \label{cor:indefinite}
Let $p\geq3$ be an odd integer. 
\begin{enumerate}
\item
$P(p, -p-2, p+4)$ is slice in $m \CP^2 \# n \overline{\CP}^2$ if and only if $(m,n) \neq (0,0)$. 
\item $P(p, -p-4, p+6)$ is slice in $m \CP^2 \# n \overline{\CP}^2$ if and only if $(m,n) \notin \{(0,0), (1,0)\}$. 
\item $P(p, -p-6, p+8)$ is slice in $m \CP^2 \# n \overline{\CP}^2$ if and only if $(m,n) \notin \{(0,0), (1,0), (2,0)\}$. 
\end{enumerate}
\end{corollary}

\begin{proof}
In each case, the knot can be changed to a ribbon knot by a negative to positive crossing change, and is therefore $H$-slice in $\ol{\CP^2}$. Now apply~\cref{cor:precise3strand} and the result of \cites{livingston-slicing, livingston-nullhomologous}, that every knot of genus one can be converted to the unknot by two generalized crossing changes (one positive to negative and one negative to positive).
\end{proof}
For other families of pretzel knots, for example $K=P(p, -p-4, p+8)$ for some $p\geq 3$, 
we have almost complete knowledge of when $K$ is slice in $m \CP^2 \# n \overline{\CP}^2$; here the only outstanding question is whether this knot is slice in $\overline{\CP}^2$ (we suspect not but have no way to prove it). In particular,  this knot has signature equal to zero and hence,  since its genus is one,  the Tristram-Levine signature function vanishes. 

We finish this section by proving a version of \cref{prop:connectedsumsoftwist} for pretzel knots using \cref{prop:equations}.
\begin{corollary}\label{cor:bigstrand}
Let $p\geq 3$ be an odd integer. Then for any $k>0$ the $(2k+1)$-strand pretzel knot $K:=P(p, -p-2, p, -p-2, \dots, p)$ has $\ucp(K)=k$.
\end{corollary}

\begin{proof}
The knot $K$ can be changed to the ribbon knot $P(p,-p,p,-p,\dots,p)$ by $k$ positive to negative crossing changes. Therefore, $\ucp(K)\leq k$. 

Assume $K$ is $H$-slice in $\#^m\CP^2$. We will show that $m\geq k$. In \cref{cor:mutant-equations},
we set $p_1 = \cdots = p_{k+1} =p$ and $q_1 = \cdots = q_k = -p-2$. Observe that for each $j$, Condition~\eqref{item:mutant-pretzel-1} together with $|p_i|=|q_j|-2$ implies that $|a_i^j|\leq 1$ for all $i$ and indeed $\{a_1^j, \cdots, a_{k+1}^j\} = \{1,0, \cdots, 0\}$ so that the first equation, $\sum_{i=1}^{k+1} a_i^j=1$, is satisifed. We also have 
\begin{equation*}
2=2 \sum_{\ell=1}^{m} (b_{\ell}^j)^2.
\end{equation*}
Thus, $\{b_1^j,b_2^j, \cdots, b_{m}^j\}=\{\pm 1, 0, \cdots, 0\}$. Given this, we see that for any $j,j'$ we must have
\begin{equation*}
\big(\sum_{i=1}^{k+1} a_i^ja_i^{j'}\big)p =0 
\end{equation*}
if there is to be any solution to \eqref{item:mutant-pretzel-2} from \cref{cor:mutant-equations}. 
\begin{equation}\label{b}
0 =\sum_{\ell=1}^{m} b_{\ell}^j  b_{\ell}^{j'}, \;\; 1 \leq j <j' \leq k. 
\end{equation}
We saw above that the set $\{b_1^j, \cdots, b_{m}^j\}=\{\pm 1, 0, \cdots, 0\}$. 
For each $j$, let the unique index $n$ such that $b_n^j \neq 0$ be called $n(j)$. By \eqref{b}, $n(j) \neq n(j')$, for all $j \neq j'$, and $\{n(j)\}_j \in \{1,2,\dots, m\}$ where $j \in \{1,\cdots, k\}$. So we need an injective map $\{1, \cdots, k\} \hookrightarrow \{1,2,\dots, m\}$. By the pigeonhole principle we need $k \leq m$, as claimed.
\end{proof}

\begin{remark}
A similar proof shows that for any odd $p\geq 3$, the $(2k+1)$-strand pretzel knot $K:=P(p, -p-2, \dots, p, -p-2, r)$ has $\ucpbar(K)=k$, whenever $r>p+2$. 
\end{remark}

\begin{remark} For the pretzel knots whose $\CP^2$-slicing numbers we compute, we find a minimal sequence of crossing changes such that each intermediate knot is again a pretzel. This feature of our result can be stated in terms the concordance knot graph, whose vertices are concordance classes of knots and whose edges represent positive to negative crossing changes~\cite{jabuka2019knot}*{Definition~1.1}.  We find that, for these knots, there is a minimal length path from the concordance class of $K$ to the class of the unknot, such that every vertex on the path is represented by a pretzel knot. To see this, assume that the concordance graph admitted a shorter path from the class of $K$ to the class of the unknot. This would imply that $\ucp(K)$ is lower than computed, since $K$ would be transformed into a slice knot by too short a sequence of crossing changes and concordances.
\end{remark} 

\section{\texorpdfstring{Comparing the topological and smooth $\CP^2$-slicing numbers}{Comparing the topological and smooth CP2-slicing numbers}}\label{sec:top-v-smooth}
We first give the proof of \cref{thm:upperbound}, which we restate below.  

\begin{reptheorem}{thm:upperbound}
Let $K\subseteq S^3$ be a knot with Seifert matrix~$A$, with respect to some Seifert surface~$F$ and choice of generators~$\alpha_1, \dots, \alpha_{2g}$ for $H_1(F;\Z)$, where $g$ is the genus of $F$.

If there exists an integral $2g \times 2g$ matrix $B$ with $\det (tB-B^T)=\pm t^k$, for some $k$, and integers $\{c_{i,j}\}$ for $i=1,\cdots, n$ and $j=1,\dots, 2g$, such that $A$ can be decomposed as the difference 
\begin{align*}
A_{2g\times 2g}= B_{2g\times 2g}
- \sum_{i=1}^n \left[
 \begin{array}{cccc}
c_{i,1}^2 & c_{i,1}c_{i,2} & \dots &c_{i,1} c_{i,2g} \\
c_{i,1} c_{i,2}  & c_{i,2}^2 & \ddots & c_{i,2} c_{i,2g}\\
\vdots & \ddots & \ddots & \vdots \\
c_{i,1} c_{i,2g} & c_{i,2} c_{i,2g} & \dots & c_{i,2g}^2
\end{array}
\right],
\end{align*}
then $\ucp^\top(K) \leq n$. 
\end{reptheorem}

\begin{proof}
Our goal is to show that $K$ is obtained from the topologically slice $K_{\gamma}$ by adding $n$ generalized positive crossings. We do this by showing that we can obtain a topologically slice knot from $K$ by adding $n$ generalized negative crossings (see~\cref{rem:negging}). 

First we find a link $\gamma:=(\gamma_1 \sqcup \cdots \sqcup \gamma_n) \subseteq S^3 \smallsetminus \nu(K)$, such that $\gamma\subseteq S^3$ is the unlink and such that $\text{lk}(\gamma_i,  \alpha_j)= c_{i,j}$ for $i=1,\dots, n$ and $j=1, \dots,  2g$. This step is straightforward: choose $\gamma$ to lie in $S^3 \smallsetminus \nu(F)$, as illustrated in~\cref{fig:gamma}, so that the desired linking number condition is satisfied. 
Note that since we are computing (signed) linking numbers we need $\gamma$ to be oriented. Given an explicit diagram for $F$ as a suface drawn in disk-band form, we directly arrange for $\gamma\subseteq S^3$ to be trivial, e.g.\ by requiring that $\gamma_j$ always passes above $\gamma_i$, for $j>i$. 

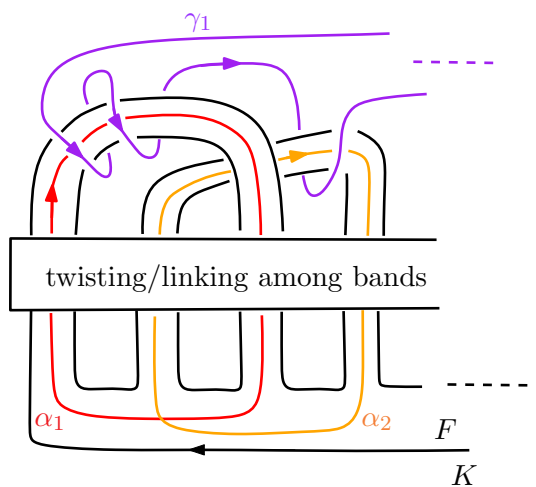
\begin{figure}[htb]
\centering
\begin{tikzpicture}
\node[anchor=south west,inner sep=0] at (0,0){\includegraphics[width=7cm]{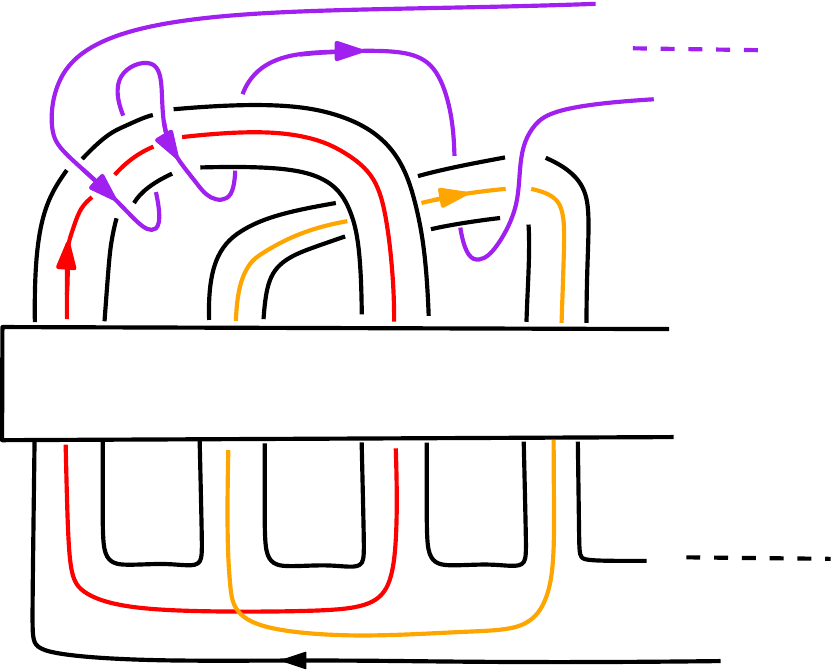}};
\node at (6,-0.25) {$K$};
\node at (5.75,0.35) {$F$};
\node at (2.5, 5.75) {\pink{$\gamma_1$}};
\node at (0.55,0.45) {\red{$\alpha_1$}};
\node at (4.85,0.45) {\orange{$\alpha_2$}};
\node at (3,2.35) {twisting/linking among bands};
\end{tikzpicture}
\caption{Proof of \cref{thm:upperbound}. The link $\gamma$ is constructed in the complement of a Seifert surface $F$ for $K$. The basis elements $\alpha_1$ and $\alpha_2$ for $H_1(F,\Z)$ are shown. In this case $c_{1,1}=2$ and $c_{1,2}=-1$. We only show one component $\gamma_1$ of~$\gamma$. }\label{fig:gamma}
\end{figure}

Let $K_\gamma$ denote the image of $K$ in the copy of $S^3$ produced by performing simultaneous $+1$-framed Dehn surgery along every $\gamma_i$. We will show next that $K_\gamma\subseteq S^3$ is topologically slice in $B^4$. Since $\gamma$ is disjoint from $F$,  under the $+1$ twists we obtain a Seifert surface for $K_{\gamma}$ in the new copy of $S^3$, which we will call $F_{\gamma}$. Moreover, we claim that the Seifert form on $F_{\gamma}$ corresponding to the generators of $H_1(F_{\gamma};\Z)$ given by the images of $\alpha_1, \dots,  \alpha_{2g}$ is exactly 
\begin{align*}
A+ \sum_{i=1}^n \left[
 \begin{array}{cccc}
c_{i,1}^2 & c_{i,1}c_{i,2} & \dots &c_{i,1} c_{i,2g} \\
c_{i,1} c_{i,2}  & c_{i,2}^2 & \ddots & c_{i,2} c_{i,2g}\\
\vdots & \ddots & \ddots & \vdots \\
c_{i,1} c_{i,2g} & c_{i,2} c_{i,2g} & \dots & c_{i,2g}^2
\end{array}
\right]= B_{2g \times 2g}.
\end{align*}
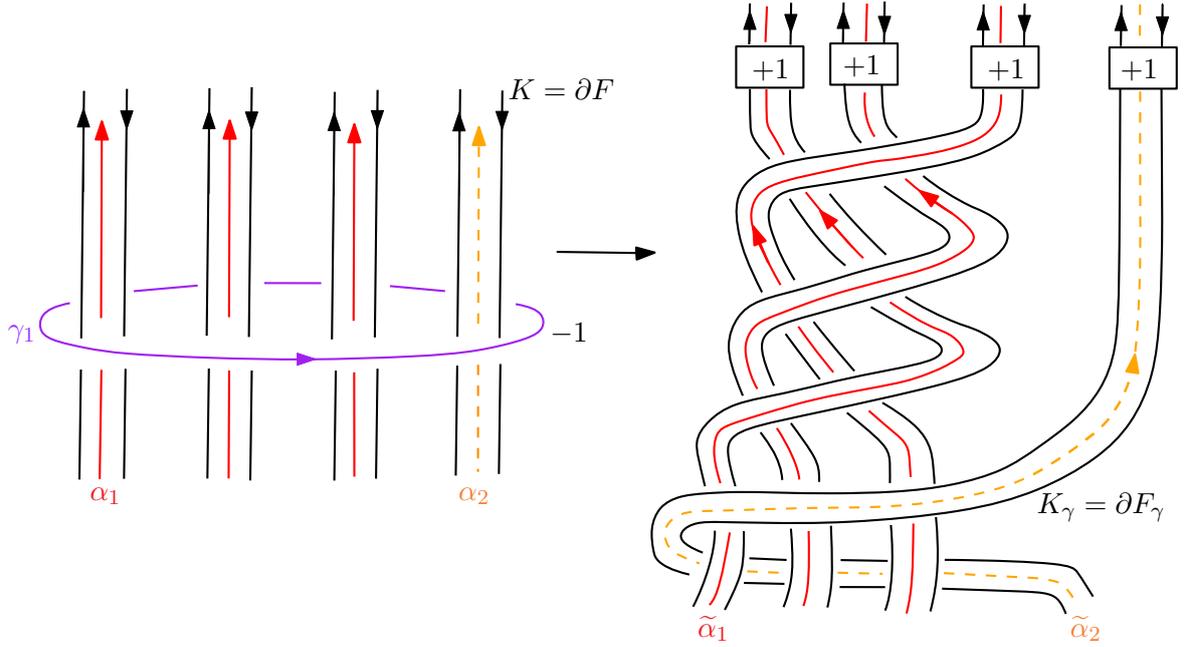
\begin{figure}[htb]
\centering
\begin{tikzpicture}
\node[anchor=south west,inner sep=0] at (0,0){\includegraphics[width=15cm]{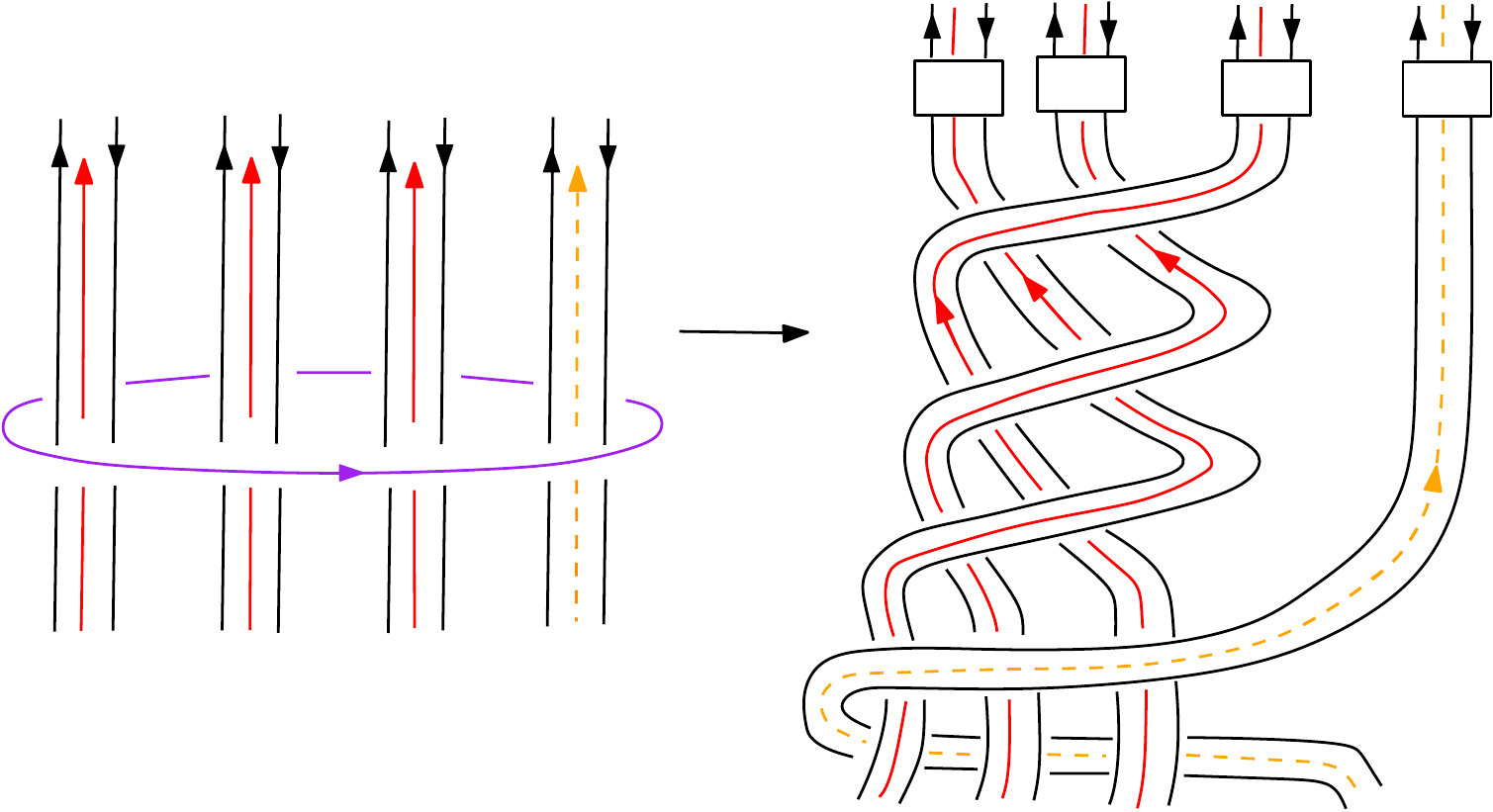}};
\node at (6.9,7) {$K=\partial F$};
\node at (-0.2,3.75) {\pink{$\gamma_1$}};
\node at (0.9,1.6) {\red{$\alpha_1$}};
\node at (5.75,1.6) {\orange{$\alpha_2$}};
\node at (7,3.75) {$-1$};
\node at (9.65,7.25) {$+1$};
\node at (10.85,7.3) {$+1$};
\node at (12.75,7.25) {$+1$};
\node at (14.5,7.25) {$+1$};
\node at (14,1.45) {$K_\gamma=\partial F_\gamma$};
\node at (8.9,-0.15) {\red{$\wt{\alpha}_1$}};
\node at (13.8,-0.15) {\orange{$\wt{\alpha}_2$}};
\end{tikzpicture}
\caption{A portion of the knot $K$ with 
a portion of 
its Seifert surface $F$ is shown on the left in black. A component $\gamma_1\subseteq \gamma$ is shown in purple, as well as basis elements $\alpha_1,\alpha_2$ for $H_1(F;\Z)$, in red and orange respectively. The specific case of $c_{1,1}=3$ and $c_{1,2}=1$ is shown. On the right we have the knot $K_\gamma$ with 
a portion of 
 its Seifert surface $F_\gamma$, in black. The images of $\alpha_1$ and $\alpha_2$ are shown, denoted by $\wt{\alpha}_1$ and $\wt\alpha_2$, again in red and orange. 
Note that the self linking of $\alpha_1$ and $\alpha_2$ has increased by $9$ and $1$ respectively. Similarly, $\lk(\wt\alpha_1,\wt\alpha_2)=\lk(\alpha_1,\alpha_2)+3$.}\label{fig:twisting}
\end{figure}
One may see this directly from looking at a diagram, such as in~\cref{fig:twisting}. Alternatively, this also follows from~\cite{hoste-casson}*{Lemma~1.1}, which provides an explicit formula to compute linking numbers of curves in an integer homology $3$-sphere obtained by performing Dehn surgery on some link. Our setup is a particular simple case of this situation. Briefly, for us the Dehn surgery is performed on an unlink and the linking framing matrix is the negative of the identity matrix. Let $Y$ denote the copy of $S^3$ obtained as the result of Dehn surgery, and let $\wt{\alpha}_1,\cdots,\wt{\alpha}_{2g}$ denote the images of the curves $\alpha_1, \dots, \alpha_{2g}$ in $Y$. Then the formula from~\cite{hoste-casson}*{Lemma~1.1} gives 
\[
\lk_Y(\alpha_j,\alpha_k^+)=\lk_{S^3}(\alpha_j,\alpha_k^+)+\left[ \begin{array}{cccc}
c_{1,j}	&c_{2,j}	&\dots 	&c_{n,j}
\end{array}\right]\cdot
\left[
\begin{array}{c}
c_{1,k}\\
c_{2,k}\\
\vdots\\
c_{n,k}
\end{array}
\right],
\]
as needed. Since $B$ is the Seifert matrix for $K_\gamma$, we compute that $\Delta_{K_{\gamma}}(t)= \det\left(tB-B^T \right) = \pm t^k$ for some $k$ by hypothesis. In other words, $K_\gamma\subseteq S^3$ has trivial Alexander polynomial, and hence is topologically slice~\cite{FQ}*{Theorem~11.7B}. In summary, $K$ is obtained from the topologically slice knot $K_\gamma$ by adding $n$ generalized positive crossings, and therefore $\ucp^\top(K)\leq n$, as claimed.
\end{proof}

We now apply \cref{thm:upperbound} to find knots with nontrivial, finite, and distinct $\ucp$ and $\ucp^\top$ as mentioned in \cref{sec:introduction}. 

\begin{proposition}\label{prop:ucp2topis1}
For any odd $p \geq 3$ and any $k \geq 1$,  the pretzel knot $K_{p,k}=P(p,-p-2k, 3p+8k-2)$ has $\ucp^{\top}(K_{p,k})=1$. 
\end{proposition}

\begin{proof}
First, observe that the determinant of $K_{p,k}$ equals
\[ \det(K_{p,k})=| p(-p-2k)+p(3p+8k-2) + (-p-2k)(3p+8k-2)| = |(4k+p)^2-4k|,\]
which is strictly larger than 1.  Therefore,   since our conditions also imply that the twist parameters of $K_{p,k}$ are both larger than 1 in absolute value and  not of the form $\{a, -a, b\}$ for some integers $a$ and $b$,  Theorem 1.5  from~\cite{miller-pretzel} implies that $K_{p,k}$ is not topologically slice and hence that  $\ucp^\top(K_{p,k}) \geq 1$. 

The standard Seifert surface $F$ for $K$ with respect to the basis $\{x,y\}$ shown in~\cref{fig:pretzel-basis} has Seifert matrix  
\[\left[ \begin{array}{cc} p+3k-1 & \frac{-p-1}{2}-k \\ \frac{-p+1}{2}-k & -k \end{array} \right]. \]
\begin{figure}[htb]
\centering
\begin{tikzpicture}
\node[anchor=south west,inner sep=0] at (0,0){\includegraphics[width=6cm]{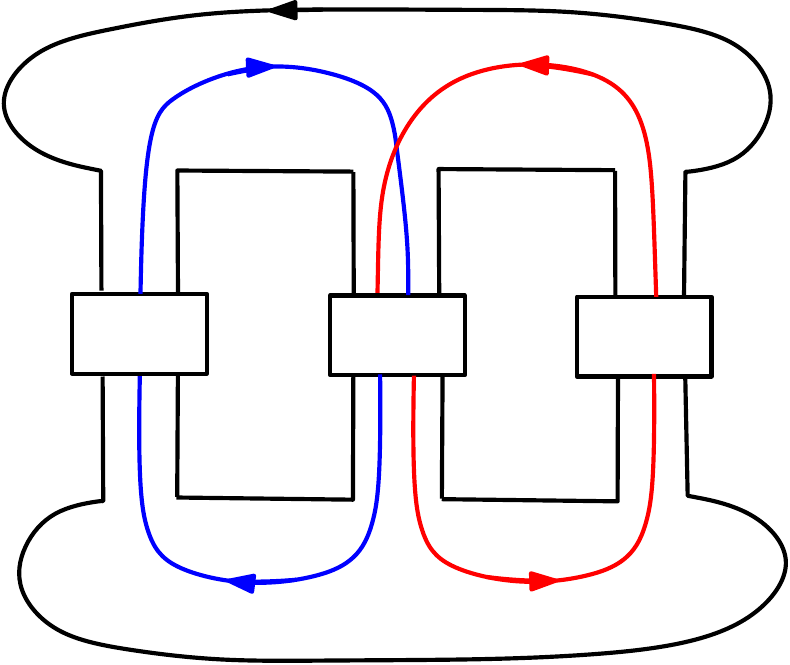}};
\node at (1,2.5) {$a$};
\node at (3,2.5) {$b$};
\node at (4.9,2.5) {$c$};
\node at (3,-0.25) {$P(a,b,c)$};
\node at (5.1,4.3) {\red{$x$}};
\node at (1,4.3) {\blue{$y$}};
\end{tikzpicture}
\caption{The $3$-strand pretzel knot $P(a,b,c)$, for $a,b,c$ odd integers, with a Seifert surface $F$. A basis $\{x,y\}$ for $H_1(F;\Z)$ is shown. The numbers in the boxes correspond to half-twists. For our examples,  $a=p$,  $b=-p-2k$, and $c=3p+8k-2$.}\label{fig:pretzel-basis}
\end{figure}
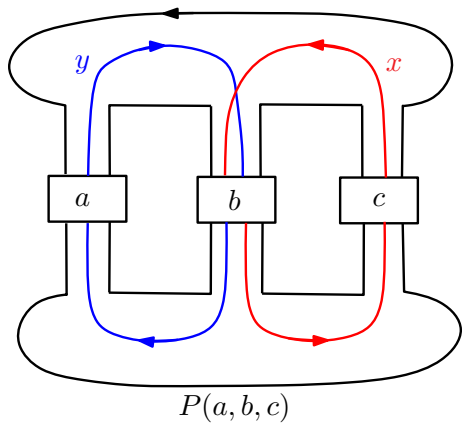
The change of basis $\{x,y\}\to \{x+y,y\}$ yields the Seifert matrix 
\[\left[ \begin{array}{cc} -1 & \frac{-p-1}{2}-2k \\ \frac{-p+1}{2}-2k & -k \end{array} \right]
= \left[ \begin{array}{cc} 0 & 0 \\ 1 & * \end{array} \right]
- \left[ \begin{array}{cc}(1)^2& 1 \cdot (\frac{p+1}{2}+2k)  \\1 \cdot (\frac{p+1}{2}+2k)   & (\frac{p+1}{2}+2k)^2 \end{array} \right],
 \]
where $*= (-k)+ (\frac{p+1}{2}+2k)^2$.  Therefore, by Theorem~\ref{thm:upperbound},  we have that $\ucp^{\top}(K_{p,k}) \leq 1$ as desired. 
\end{proof} 

As an application of the above proposition, we obtain the following corollary. 

\begin{repcorollary}{thm:top-smooth-distinct}
Let $p\geq 3$ be odd.
\leavevmode
\begin{enumerate}
\item The pretzel knot $K_{p,2}:=P(p, -p-4, 3p+14)$ has $\ucp^\top(K_p)=1$ and $\ucp(K)=2$. Moreover, the elements of $\{K_{p,2}\}_p$ are distinct in topological concordance.  
\item The pretzel knot $K_{p,3}:=P(p, -p-6, 3p+22)$ has $\ucp^\top(K_p)=1$ and $\ucp(K)=3$. Moreover, the elements of $\{K_{p,3}\}_p$ are distinct in topological concordance.   
\end{enumerate}
\end{repcorollary}

\begin{proof}
The computation of $\ucp^\top$ comes from \cref{prop:ucp2topis1}, while that of $\ucp$ comes from \cref{cor:precise3strand}. 

Recall that the Alexander polynomial 
\[
\Delta_{P(p,q,r)}(t)=\frac{(pq+qr+pr)(t-2+t^{-1})+(t+2+t^{-1})}{4},\]
for odd integers $p,q,r$. Since $K(p,k)=P(p,-p-2k,3p+8k-2)$, we therefore have that
\begin{equation}\label{eq:alex-poly-pretzel}
\Delta_{K(p,k)}(t)=\frac{(4k-(4k+p)^2)(t-2+t^{-1})+(t+2+t^{-1})}{4}.
\end{equation}
So $\det{K_{p,k}}= (4k+p)^2-4k$. Suppose that  $(4k+p)^2-4k=n^2$ for some $n\geq1$. Then 
\[
(4k+p+n)(4k+p-n)=4k.
\] 
Here $4k+p+n$ and $4k$ are strictly positive integers, so $4k+p-n\geq 1$ as well. But then the left hand side is strictly larger than $4k$, which is a contradiction. So $\det(K_{p,k})$ is not a square, and hence $\Delta_{K_{p,k}}(t)$ does not factor as $f(t)f(t^{-1})$ for any $f(t) \in \mathbb{Z}[t^{\pm1}]$.  Since $\Delta_{K_{p,k}}(t)$ is a symmetric degree two polynomial,  it must further be irreducible. 
Moreover, from~\eqref{eq:alex-poly-pretzel} it is clear that for a fixed $p$, we have $\Delta_{K_{p,k}}(t)=\Delta_{K_{p,k'}}(t)$ if and only if $k=k'$.  

Therefore, for any $k \neq k'$, 
\[\Delta_{K_{p,k} \#-K_{p,k'}}(t)= \Delta_{K_{p,k}}(t) \Delta_{K_{p,k'}}(t)\]
is a product of two distinct irreducible degree two polynomials, and hence cannot factor as $g(t)g(t^{-1})$ for any $g(t) \in \mathbb{Z}[t^{\pm1}]$.  It follows that $K_{p,k} \#-K_{p,k'}$ is not even algebraically slice, and consequently $K_{p,k}$ and $K_{p,k'}$ are distinct in topological concordance.
\end{proof}

\begin{remark}
As mentioned after the proof of \cref{cor:precise3strand}, we conjecture that $\ucp(P(p, -p-2k, r))=k$, whenever $k\geq 0$ and for odd integers $p$ and $r$, with $p\geq 3$ and $r>p+2k$ (see~\cref{prob:genus-one}). If the conjecture is true, then the proof above shows that $\{K_{p,k}\}_p$ is an infinite family of knots with $\ucp^\top=1$ and $\ucp=k$ (see \cref{prob:top-smooth}).
\end{remark}

We finish this section by giving the proof of \cref{prop:genus-bound}, which we first restate.

\begin{repproposition}{prop:genus-bound}
Let $K$ be a knot with Seifert genus one. 
If $\sigma(K)=2$, then $\ucp^\top(K)= \infty$.  Otherwise, $\ucp^\top(K) \leq 4$. 
\end{repproposition}
\begin{proof}
The first assertion follows from~\cite{CHH}*{Proposition~1.2}. For the second assertion, assume that $K$ is a knot with genus one and non-positive signature. Choose some basis for $H_1(F;\Z)$, and let the corresponding Seifert matrix be 
\[A:=\left[ \begin{array}{cc} a & b+1 \\ b & c \end{array} \right].\] 
Note that if $a=b=0$ then $\Delta_K(t)=1$ and so $\ucp^\top(K)=0$, so we can assume that is not the case.

\begin{claim} For any Seifert surface $F$ for $K$ with genus one there exists a simple closed curve $\gamma$ on $F$ such that the $F$-induced framing of $\gamma$, denoted $\fr(\gamma)$, is strictly negative. 
\end{claim}

\begin{proof}[Proof of claim]
If $c<0$ then the claim is proved, so assume that $c \geq 0$. Consider the matrix $A+A^T$. By definition $\det(A+A^T)=4ac-(2b+1)^2$ equals the determinant of $K$ and therefore cannot be zero. On the other hand, if $4ac-(2b+1)^2>0$ then $4ac>0$ and therefore $c>0$. Then since $\det(A+A^T)$ and $2c$ are both positive, the matrix $A+A^T$ is positive definite by Sylvester's criterion, implying that $\sigma(K)=2$, a contradiction. 
Therefore, we have that $4ac-(2b+1)^2 <0$.

In terms of our current basis for $H_1(F;\Z)\cong \Z\oplus \Z$, we have that  
\begin{equation}\label{eq:selflinking}
\fr(x,y):=\lk((x, y), (x,y)^+)= ax^2+ (2b+1) xy + cy^2.
\end{equation}
In each of the following cases, consider the homology class $(x,y)$ indicated below in~\cref{table}. The self-linking number in each case is negative, as computed using~\eqref{eq:selflinking}. 

\begin{table}[h!]
\centering
\caption{Proof of \cref{prop:genus-bound}.}\label{table}
\begin{tabular}{cccc}\toprule
Case no.\	&	&$(x,y)$	&$\fr(x,y)$\\
	\midrule
(1)	&$a<0$	&$(1,0)$	&$a$\\
(2)	&$a=0,b>0,c>0$	&$(-c,1)$	&$-2bc$\\
(3)	&$a=0,b>0, c=0$	&$(-1,1)$	&$-(2b+1)$\\	
(4)	&$a=0,b<0,c>0$	&$(2c,1)$	&$(4b+3)c$\\
(5)	&$a=0,b<0, c=0$	&$(1,1)$	&$(2b+1)$\\	
(6)	&$a>0$	&$(-2b-1,2a)$	&$a(4ac-(2b+1)^2)$\\
\bottomrule
\end{tabular}
\end{table}
In cases (1-5), the class $(x,y)$ is primitive, so choose $\gamma$ to be a simple closed curve representing it. Case (6) is more interesting. While $(-2b-1, 2a)$ may not be primitive itself, we can choose some primitive element $(x',y')\in \Z\oplus\Z$ such that $(-2b-1, 2a)=n(x',y')$ for some integer $n$. Then $\fr(-2b-1, 2a)=n^2\cdot \fr(x',y')$ and so $\fr(x',y')$ must also be negative. Then choose $\gamma$ to be a simple closed curve representing the class $(x',y')$. This completes the proof of the claim.
\end{proof}

We have now found a simple closed curve $\gamma\subseteq F$ with $\fr(\gamma)<0$. Since $[\gamma]$ is primitive, we can find some curve $\delta\subseteq F$ where the algebraic intersection $\gamma\cdot \delta=-1$ and the classes of $\gamma$ and $\delta$ form a basis for $H_1(F;\Z)$. In other words, $\{\gamma,\delta\}$ is a symplectic basis for $H_1(F;\Z)$. 
Let $B:= \left[ \begin{array}{cc} p & q+1 \\ q & r \end{array} \right]$ be the Seifert matrix of $F$ with respect to the basis $\{\gamma,\delta\}$, and note that we have that $p<0$ by construction. 

\begin{remark}
We can now easily prove that $\ucp^\top(K) \leq 5$.  Use Lagrange's four-square theorem (see e.g.\ \cite{serre-book}*{p.\ 47, Corollary~1}) to write $0\leq -p-1= k^2+ \ell^2+ m^2 + n^2$, for some integers $k,\ell, m, n$, and observe that 
\begin{align*}
\left[ \begin{array}{cc} p & q+1 \\ q & r \end{array} \right]
=\left[ \begin{array}{cc} 0 & 1 \\ 0 & v \end{array} \right]
-\left[ \begin{array}{cc} 1 & u  \\ u & u^2 \end{array} \right]
-\left[ \begin{array}{cc} k^2 & k \\ k & 1 \end{array} \right]
-\left[ \begin{array}{cc} \ell^2 & \ell \\ \ell & 1 \end{array} \right]
-\left[ \begin{array}{cc} m^2 & m \\ m & 1 \end{array} \right]
-\left[ \begin{array}{cc} n^2 & n \\ n & 1 \end{array} \right]
\end{align*}
where $u=-q-k-\ell-m-n$ and $v=r+u^2+4$. Then by~\cref{thm:upperbound}, $\ucp^\top(K)\leq 5$. But we would like to improve this bound, which we do next.
\end{remark}

\begin{claim}
There exists a simple closed curve $\alpha\subseteq F$ so that $-\fr(\alpha)-1$ can be written as a sum of three squares.
\end{claim}

\begin{proof}[Proof of claim]
According to Legendre's three-square theorem (see e.g.\  ~\cite{serre-book}*{pp.\ 45-47}), a positive integer can be written as a sum of three squares if and only if it is \textit{not} of the form $4^a(8b+7)$ for some $a,b \geq 0$. In particular, it will suffice for us to find $\alpha$ so that 
\[
-\fr(\alpha)-1\not\equiv 0 \text{ or }3 \mod{4}\]
 or, equivalently, 
 \[
 \fr(\alpha)\not\equiv 3 \text{ or } 0\mod{4}.\]

With respect to our current basis $\{\gamma,\delta\}$, we have  
\[ \fr(x,y):= \lk((x,y), (x,y)^+)= px^2+(2q+1)xy + r y^2,\]
where $p<0$. We have $\fr(1,0)=p$ so if $p\not\equiv 0$ or $3\mod{4}$, choose $\alpha$ to be a representative of $(1,0)$ and the claim is proved. Otherwise, assume that $p\equiv 0$ or $3 \mod{4}$. Then consider the primitive class $(x,2)$ with $x\equiv 1\mod{4}$.

Note that we can choose $x$ large enough so that $\fr(x,2)=px^2+2(2q+1)x+4r\leq -1$. For example, this follows from the fact that $\frac{d^2}{dx^2}\fr(x,2)=2p<0$. Assume we have chosen such an $x$, so that $-\fr(x,2)-1>0$. Now we compute 
\begin{eqnarray*}
\fr(x,2)=px^2+2(2q+1)x+4r	&\equiv p+2\mod{4}\\
	&\equiv 2 \text{ or } 1\mod{4},
\end{eqnarray*}
since $p\equiv 0$ or $3 \mod{4}$.
In particular, $\fr(x,2)\not\equiv 0$ or $3\mod{4}$ as needed. Let $\alpha$ be a representative of the class $(x,2)$. This completes the proof of the claim.\end{proof} 

Returning to the proof of the proposition, we now have a simple closed curve $\alpha\subseteq F$ such that $-\fr(\alpha)-1=k^2+\ell^2+m^2$, for some integers $k,\ell,m$. Since $[\alpha]$ is primitive, we can find some curve $\beta\subseteq F$ so that $\{\alpha,\beta\}$ form a symplectic basis for $H_1(F;\Z)$. Let $A:=\left[\begin{array}{cc}
e	&g+1\\	
g	&h
\end{array} \right]$ denote the Seifert matrix of $F$ corresponding to the basis $\{\alpha,\beta\}$. In particular, we have $e = \fr(\alpha)$. 
Observe that 
\begin{align*}
\left[ \begin{array}{cc} e & g+1 \\ g & h \end{array} \right]
=\left[ \begin{array}{cc} 0 & 1 \\ 0 & v \end{array} \right]
-\left[ \begin{array}{cc} 1 & u  \\ u & u^2 \end{array} \right]
-\left[ \begin{array}{cc} k^2 & k \\ k & 1 \end{array} \right]
-\left[ \begin{array}{cc} \ell^2 & \ell \\ \ell & 1 \end{array} \right]
-\left[ \begin{array}{cc} m^2 & m \\ m & 1 \end{array} \right]
\end{align*}
where $u=-g-k-\ell-m$ and $v=h+u^2+3$. Then by~\cref{thm:upperbound}, $\ucp^\top(K)\leq 4$, as claimed.
\end{proof}

\bibliographystyle{alpha}
\def\MR#1{}
\bibliography{bib}
\end{document}